\documentclass{amsart}

\usepackage{amsfonts}
\usepackage{fancyhdr}
\usepackage{verbatim}
\usepackage{graphicx,color}
\usepackage{mathscinet}

\usepackage{amssymb,amsfonts}
\usepackage{xcolor}

\newtheorem{theorem}{Theorem}[section]

\newtheorem{lemma}[theorem]{Lemma}
\newtheorem{lem}[theorem]{Lemma}
\newtheorem{proposition}[theorem]{Proposition}

\newtheorem{corollary}[theorem]{Corollary}

\newtheorem{cor}[theorem]{Corollary}
\newtheorem{definition}[theorem]{Definition}

\newtheorem{example}[theorem]{Example}
\newtheorem{remark}[theorem]{Remark}

\newtheorem{Example}[theorem]{Example}

\newtheorem{Theorem}[theorem]{Theorem}

\newcommand{\be}{\begin{equation}}
\newcommand{\ee}{\end{equation}}
\newcommand{\beq}{\begin{equation*}}
\newcommand{\eeq}{\end{equation*}}
\newcommand{\enq}{\end{equation}}
\newcommand{\ben}{\begin{eqnarray}}
\newcommand{\een}{\end{eqnarray}}
\newcommand{\bea}{\begin{eqnarray*}}
\newcommand{\eea}{\end{eqnarray*}}

\newcommand{\At}{ {\widetilde{A}}}

\def\Gammat{{\widetilde{\Gamma}}}

\newcommand{\Hc}{ {\mathcal{H}}}

\def\cH{{\mathcal H}}

\newcommand{\cL}{ {\mathcal{L}}}
\newcommand{\cC}{ {\mathcal{C}}}
\newcommand{\cA}{{\mathcal{A}}}

\newcommand{\zb}{\overline{z}}

\newcommand{\Et}{ {\widetilde{E}}}

\newcommand{\clos}{\mbox{\rm clos}}

\def\ker{{\mathrm{ker\,}}}
\def\Ran{{\mathrm{Ran\,}}}
\def\rank{{\mathrm{rank\,}}}
\def\Span{{\mathrm{Span\,}}}

\newcommand{\Cc}{{\mathbb{C}}}

\newcommand{\la}{\lambda}

\newcommand{\mut}{\widetilde{\mu}}
\newcommand{\mub}{\overline{\mu}}
\newcommand{\mutb}{\overline{\widetilde{\mu}}}

\newcommand{\llangle}{\left\langle}
\newcommand{\rrangle}{\right\rangle}

\def\C{\mathbb C}
\def\R{\mathbb R}
\def\N{\mathbb N}

\newcommand{\norm}[1]{\left\Vert#1\right\Vert}

\newcommand{\twovec}[2]{\left(\begin{array}{c} #1 \\  #2 \end{array}\right)}

\newcommand{\U}{\left(\begin{array}{c} v_{-} \\ u \\ v_{+} \end{array}\right)}
\newcommand{\Uzero}{\left(\begin{array}{c} v_{-} \\ 0 \\ v_{+} \end{array}\right)}
\newcommand{\Ut}{\left(\begin{array}{c} \widetilde{v_{-}} \\ \widetilde{u} \\ \widetilde{v_{+}} \end{array}\right)}
\newcommand{\W}{\left(\begin{array}{c} f \\ w \\ g \end{array}\right)}
\newcommand{\Hvec}{\left(\begin{array}{c} h_{-} \\ h_0 \\ h_{+} \end{array}\right)}
\newcommand{\Hzero}{\left(\begin{array}{c} h_{-} \\ 0 \\ h_{+} \end{array}\right)}

\numberwithin{equation}{section}

\def\C{\mathbb C}
\def\R{\mathbb R}
\def\N{\mathbb N}

\def\H{\mathcal H}
\def\L{\mathcal L}

\newcommand{\twomat}[4]{\left(\begin{array}{cc} #1 & #2 \\  #3 & #4 \end{array}\right)}
\newcommand{\threemat}[9]{\left(\begin{array}{ccc} #1 & #2 &  #3 \\ #4 & #5 & #6\\ #7 & #8 &#9 \end{array}\right)}

\title[Functional model]{The functional model for maximal dissipative operators:\\ An approach in the spirit of {operator knots}}

\author[Brown]{Malcolm Brown}
\address{School of Computer Science, Cardiff University, Queen's Buildings, 5 The Parade, Cardiff CF24 3AA, UK}
\email{Malcolm.Brown@cs.cardiff.ac.uk}

\author[Marletta]{Marco Marletta}
\address{School of Mathematics, Cardiff University, Senghennydd Road, Cardiff CF24 4AG, UK}
\email{MarlettaM@cardiff.ac.uk}

\author[Naboko]{Serguei Naboko}
\address{Department of Math.~Physics, Institute of Physics, St.~Petersburg State University, 1 Ulianovskaia, St.~Petergoff, St.~Petersburg, 198504, Russia}
\email{sergey.naboko@gmail.com}

\author[Wood]{Ian Wood}
\address{School of Mathematics, Statistics and Actuarial Sciences,
 University of Kent, Canterbury, CT2 7FS, UK}
\email{i.wood@kent.ac.uk}

\thanks{{
M.~Marletta and S.N.~Naboko gratefully acknowledge the support of the Leverhulme
Trust, grant RPG167, and of the Wales Institute of Mathematical
and Computational Sciences. S.N.~Naboko also was partially supported by the grant   RFBR 16-11-00443a, as well as the EC Marie Curie grant PIIF-GA-2011-299919.
}}

\begin{document}
\begin{abstract}{
In this article we develop a functional model for a general maximal dissipative operator. We construct the selfadjoint dilation of such operators. Unlike previous functional models, our model is given explicitly in terms of parameters of the original operator, making it more useful in concrete applications. }

{
 For our construction we introduce an abstract framework for working with a maximal dissipative operator and its anti-dissipative adjoint and make use of the \v{S}traus characteristic function  in our setting. Explicit formulae are given for the selfadjoint dilation, its resolvent, a core and the completely non-selfadjoint subspace; minimality of the dilation is shown. The abstract theory is illustrated by the example of a Schr\"odinger operator on a half-line with dissipative potential, and boundary condition and connections to existing theory are discussed. 
}
\end{abstract}
\maketitle

\section{Introduction}

{
In recent years, the spectral and scattering properties of non-selfadjoint problems have become a subject of much
mathematical  and physical interest.
This is the natural setting for many important problems in physics including dissipative problems (where the system loses energy), problems in hydrodynamics and   the study of metamaterials 
where progress has been driven in part by   the development   and feasibility of manufacture of novel materials with unexpected properties.  
Dissipation, at the atomic level,  plays an essential part in many processes, see for example Milton et al.~on cloaking in the presence of a superlens \cite{MNMP05}, 
Weder et al.~\cite{DIW14}  on plasma heating  through tunneling effects in   tokamaks, the work of Figotin and Welters on dissipation in composite  materials \cite{FW12}, Cherednichenko et al.~on quantum graphs using the functional model \cite{CKS} and Fr\"ohlich et al.~on scattering for the Lindblad equations \cite{FFFS17}  where dissipative methods were used. 
}

{
Mathematically these problems pose a challenge, as apart from rather exceptional cases, the well-developed methods used to examine the spectrum of  selfadjoint problems are not applicable.  
  According to Mark Krein the spectral theorem in the selfadjoint case highlights the relationship between the spectral analysis of the operator and the geometry of the Hilbert space; in contrast,
	in the spectral analysis of non-selfadjoint operators   this geometric relationship plays a much reduced role and is replaced by complex analysis.
	A tool  more     appropriate   to analyse the spectrum of non-selfadjoint operators has to be used; such a tool is the  
functional model.
This
reduces the spectral analysis of a non-selfadjoint operator to a problem in complex analysis: the canonical factorisation of the characteristic function {as an analytic} operator-valued function in the upper half-plane  (M. Liv\v{s}ic theorem). 
  The functional model  provides a systematic approach to studying the spectral and scattering theory of non-selfadjoint problems with  wide  applicability. 
}

{
Functional models were first introduced for contractions by Sz.~Nagy and Foias (see \cite{SFBK10,Nik86} and references therein)   to analyse the structure of contractions and relations between an operator, its spectrum and its characteristic function, and simultaneously, {in a different form,} by de Branges \cite{dB68}.
Since then functional models have been developed further including a very useful symmetric version of the Sz.-Nagy-Foias model due to Pavlov \cite{Pav75}.
They have been used to obtain many results in mathematical physics and in spectral analysis with applications to problems such as
Schr\"{o}dinger operators with complex potentials  and non-selfadjoint boundary conditions, and   stochastic quantum dynamics. 
 Pavlov's work on quantum switches \cite{Pav02} and Naboko and Romanov's work on time asymptotics for the  Boltzmann operator \cite{NR01}
  have relied heavily on it. 
The  best known application is Lax-Phillips scattering theory which corresponds to a special case of the Sz.~Nagy-Foias functional model  when the characteristic function of the operator to be studied is an inner function (this excludes the possibility of absolutely continuous spectrum).
Functional models can be used to find conditions for the existence and completeness for wave operators in scattering theory, the scattering matrix and spectral shift function and give explicit formulae for them in the framework of the model (see \cite{Nab81}). Moreover, the functional model has applications in inverse scattering theory and can help provide information on which part of the operator can be reconstructed from measurements and which part can be cloaked, {see e.g.~\cite{LP67}.}
}

{
Constructions of functional models often rely on abstract results, making it  hard to apply the results to specific examples. Based on Pavlov's explicit construction of the functional model for Schr\"{o}dinger operators in \cite{Pav77}, the functional model
for the situation of a family of additive perturbations to a selfadjoint operator was developed in \cite{Nab81}. The approach has several concrete advantages: It gives explicit formulae for all expressions arising in the model and, by developing the spectral form of the functional model, it importantly also allows the study of non-dissipative operators within the framework of the functional model.  
In further work, Ryzhov has developed a functional model for the case when the perturbation is only in the boundary conditions \cite{Ryz07}. 
}

{
 In this paper, we consider  a general maximal dissipative operator and   develop the so-called `translation form' of the functional model. This consists of constructing a selfadjoint dilation of the maximal dissipative operator by adding two so-called channels. The channels model incoming and outgoing effects and the dilation acts in them as a first order differential operator. 
Kudryashov \cite{Kud82} presented an explicit form of the selfadjoint dilation of a  general maximal dissipative operator, announcing only the result without any details of the proof. Later, Ryzhov \cite{Ryz08} reconstructed this form of the selfadjoint dilation and gave a complete proof. In both cases the explicit construction of the dilation is purely abstract and contains the {square roots of some complicated operators which can only be explicitly calculated in a useful form} in the case of a rank one perturbation of a selfadjoint operator.
}

{Our goal in this paper is to present an alternative construction of the selfadjoint dilation of a maximal dissipative operator  
 using an abstract construction in the spirit of {operator knots \cite{Bro71}. One of the main ingredients of that theory is a Lagrange identity which we here extend from the case of bounded operators to our more general setting. This relates functions in the domain of the maximal dissipative operator via so-called $\Gamma$-operators to the non-selfadjoint part of the operator. 
The approach is also inspired by the theory of boundary triples (see, e.g.~\cite{BMNW08,DM91, GG91, Koc75}), where the Lagrange identity is replaced by a 
 Green formula which relates functions in the domain of an operator with their boundary data. }
  In our view, the flexibility of the choice of the $\Gamma$-operators in the Lagrange identity gives a serious advantage to our construction. It will enable us to choose the $\Gamma$-operators so that expressions arising in the dilation can be given explicitly in terms of model parameters such as coefficients or boundary conditions of the maximal dissipative operator. This will allow us to obtain spectral results in explicit terms of these model parameters which is not possible in the constructions in \cite{Kud82} and \cite{Ryz08}.
e roots of operators are involved.}

{
This paper is structured in the following way: Section 2 gives a brief introduction to the main facts about  dissipative operators and selfadjoint dilations and an overview of the historical development.
The Sections 3-4 discuss background material and introduce concepts and notation required. Although the ideas in these sections may not be completely new, many of them can already be found in the work of \v{S}traus for special cases, to the best of our knowledge the presentation in full generality here is new and additionally serves to make our presentation self-contained. 
 In Section 3 we introduce an abstract framework for working with a maximally dissipative operator and its anti-dissipative adjoint. The \v{S}traus characteristic function is introduced in our framework in Section 4 and its relation to the Sz.-Nagy-Foias characteristic function for contractions is  discussed. 
The Sections 5-8 cover the main new material of this paper.
Section 5 introduces our formula for the selfadjoint dilation and its domain. This is illustrated for two examples, including a dissipative Schr\"odinger operator on a half-line, in Section 6, showing the advantages of our explicit construction in the spirit of operator knots. {Properties of the dilation are shown in Section 7, in particular, we calculate the resolvent and show minimality of the dilation. As a by-product we present a relatively simple expression for the dilation on a core. We also discuss  complete non-selfadjointness of the dilation, giving an explicit formula for the completely non-selfadjoint subspace.} Finally, in Section 8, we discuss connections of our construction with the Kudryashov/Ryzhov model and look more closely at the characteristic function in the case of a symmetric minimal operator, {even in the case of differing deficiency indices}.
}

{The paper develops the translation form of the functional model; the so-called spectral form where it is represented as the operator of multiplication by an independent variable in
some auxiliary vector-valued function space will be the topic of a forthcoming paper.}

\section{Basic properties of dissipative operators}
This section reviews some classical results on dissipative operators. For more on the subject, we refer the reader to \cite{HP57,LP67,SFBK10}.
\begin{definition}
A densely defined linear operator $A$ with domain $D(A)$ in $H$ is called dissipative if $\Im\llangle Ah,h\rrangle \geq 0$ for all $h\in D(A)$. $A$ is called anti-dissipative if $(-A)$ is dissipative.
\end{definition}

The next result shows that every $\lambda\in\C^-$ is of regular type for a dissipative operator $A$.

\begin{proposition} Let $A$ be dissipative.
For any $f\in D(A)$ and $\lambda\in\C^-$ we have $$\norm{(A-\lambda)f}\geq |\Im\lambda|\norm{f}.$$
\end{proposition}

 One would like to additionally have invertibility of $A-\lambda$, i.e.~$\mathrm{Ran}(A-\lambda)=H$.  The following theorem guarantees that every densely defined dissipative operator has an extension with this property:

\begin{proposition}[R.~Phillips] For any densely defined dissipative operator $A$ there exists at least one dissipative extension $\At$, i.e.~$D(A)\subseteq D(\At)$ and $\At\vert_{D(A)}=A$, such that $\mathrm{Ran}(\At-\lambda)=H$ for all $\lambda\in\C^-$.
\end{proposition}

\begin{definition}
Dissipative operators which have no non-trivial dissipative extensions are called \textit{maximal dissipative operators} (MDO).
\end{definition} 

MDOs are characterised by the existence of a bounded resolvent $(A-\lambda)^{-1}$ on the whole of $H$ for $\lambda\in\C^-$, i.e.~MDOs have no spectrum in the lower half plane. A dissipative operator may have several maximal dissipative extensions.

There is a bijective map between the class of MDOs and contractions which is given by the Cayley transform: For any MDO, $(A+i)^{-1}\in B(H)$. Define
\be\label{Cayley} T=I-2i(A+i)^{-1} = (A-iI)(A+iI)^{-1}.\ee
This is an operator version of the M\"{o}bius transform.

Properties of the Cayley transform:
\begin{enumerate}
	\item Let $T$ be the Cayley transform of an MDO. Then $D(T)=H$ and $\norm{T}\leq 1$.
	\item The Cayley transform is a one-to-one  map from the class of MDOs in $H$ onto the class of all contractions satisfying the extra condition $1\notin \sigma_p(T)$. Here, $\sigma_p(T)$ denotes the set of eigenvalues of $T$. The condition $1\notin \sigma_p(T)$ is equivalent to $\Ran(T-I)$ being dense in $H$.
	\item $\lambda\in\sigma(A)\subseteq\overline{\C^+}$ iff $(\lambda-i)/(\lambda+i)\in\sigma(T)$.
\end{enumerate}

\begin{remark}
At first glance the Cayley transform looks like a very convenient tool, replacing the complicated class of unbounded MDO by the class of contractions. However, for a  particular MDO $A$ there is rarely sufficient information on $(A+i)^{-1}$ to explicitly obtain $T$. Therefore, the importance of the Cayley transform is often purely theoretical.
\end{remark}

The real eigenvalues of an MDO exhibit the same behaviour as the eigenvalues of selfadjoint operators.
\begin{proposition}[Sz.Nagy]
Let $A$ be an MDO. Then the eigenvectors corresponding to real eigenvalues are orthogonal to all eigenvectors corresponding to different eigenvalues (real or complex). Moreover, the subspace spanned by all eigenvectors corresponding to real eigenvalues belongs to the selfadjoint subspace $H_1$ in the Langer decomposition (see Proposition \ref{prop:Langer}).
\end{proposition}

\begin{remark}
Note that for an MDO $A$  there cannot be any root vectors corresponding to real spectrum. Just as in the well-known case of matrices, this follows from the resolvent estimate $\norm{(A-\lambda I)^{-1}}\leq (|\Im\lambda|)^{-1}$ for $\lambda\in\C^-$.
\end{remark}

The proposition says that our operator consists of a part (corresponding to the set of eigenvectors of the real point spectrum) which looks like a selfadjoint operator and a remaining part. It seems reasonable to try to study the two parts separately. This idea leads to the introduction of the notion of completely non-selfadjoint operators (corresponding to the remaining part of the operator).

\begin{definition}
Let $A$ be an operator on a Hilbert space $H$, $H_1\subseteq H$ a subspace
and $P_{H_1}$ the orthogonal projection of $H$ onto $H_1$.
The subspace $H_1$ 
 is \textit{invariant} with respect to $A$ if $P_{H_1}D(A)\subseteq D(A)$ and $AP_{H_1}h\in H_1$ for all $h\in D(A)$.
It is a \textit{reducing subspace} for $A$ if both $H_1$ and $H\ominus H_1$ are invariant with respect to $A$. 
\end{definition}

\begin{definition} Let $A$ be an MDO.  $A$ is \textit{completely non-selfadjoint (cns)} if there exists no reducing subspace $H_1\subseteq H$ such that $A\vert_{H_1}$ is selfadjoint.
\end{definition}

{The following result gives an explicit formula for the completely non-selfadjoint part of the operator. {In the case of relatively bounded imaginary part the formula is simple. For more general situations the formula involves operators $\Delta$ and $\Delta_*$ which are regularisations of the (possibly non-existing) imaginary part of the operator. In our setting, we will determine an explicit formula for the completely non-selfadjoint part of an MDO in Theorem \ref{thm:cns}.}} 

\begin{proposition}\label{prop:Langer} (Langer decomposition, see \cite{Lan61,Nab81}).
Let $A$ be an MDO. Then there exists a unique decomposition of $H=H_1\oplus H_2$ into an orthognal sum of two reducing subspaces for $A$ such that $A\vert_{H_1}$ is selfadjoint in $H_1$ and  $A\vert_{H_2}$ is completely non-selfadjoint in $H_2$.

{Define
\ben\label{Delta}
\Delta&=& I-T^*T \ =\ 2i\left[(A+i)^{-1}-(A^*-i)^{-1}+2i (A^*-i)^{-1}(A+i)^{-1}  \right],
\een
\ben\label{Deltastar}
\Delta_*&=& I-TT^* \ =\  2i\left[(A+i)^{-1}-(A^*-i)^{-1}+2i (A+i)^{-1}(A^*-i)^{-1} \right]
\een
and set
$$\mathcal{M}:= \Ran(\Delta) +\Ran(\Delta_*)\subseteq H.$$
Then the completely non-selfadjoint part $H_2$ is given by
$$H_2 = \clos\left(\mathrm{Span}_{\Im\lambda< 0} \{ (A-\lambda)^{-1}\mathcal{M}\}+\mathrm{Span}_{\Im\lambda> 0} \{ (A^*-\lambda)^{-1}\mathcal{M}\} \right).
$$}

If $A$ has relatively bounded imaginary part, i.e.~$A=L+iV$ with $L=L^*$, $V\geq 0$, $V$ relatively $L$-bounded, then there is a simple explicit expression for the completely non-selfadjoint part $H_2$:
$$H_2 = \clos\left(\mathrm{Span}_{\Im\lambda\neq 0} \{ (L-\lambda)^{-1}\Ran V\} \right) 
= \clos\left(\mathrm{Span}_{\lambda\notin(\sigma(A)\cup\R)} \{ (A-\lambda)^{-1}\Ran V\} \right),$$
i.e.~$H_2$ is generated by the range of the imaginary part $V$ developed by the resolvent of the operator $A$ or its real part $L$. Moreover, $A\vert_{H_1}=L\vert_{H_1}$.
\end{proposition}

In systems theory, MDOs are used to describe systems with a loss of energy, while Hermitian operators describe systems with energy conservation. This naturally leads to the idea of including a dissipative system in a larger conservative one, taking into account `where' the energy is leaking to. The mathematical realization of this idea is due to the Hungarian mathematician B.~Sz.-Nagy in the late 50ies, but its roots go back to earlier papers by M.~Naimark. Actually, Sz.-Nagy worked with contractions rather than MDOs. However the two formulations are equivalent via the Cayley transform.


\begin{proposition}[Sz.-Nagy]\label{dilation}
 For any MDO $A$ on a Hilbert space $H$ there exists a selfadjoint operator $\L$ on a Hilbert space $\H\supseteq H$ such that 
 \be\label{dilsem}
e^{it A}=P_H e^{it\L}\vert_H,\ t\geq 0 \quad \hbox{ or equivalently } \quad (A-\lambda)^{-1}=P_H (\L-\lambda)^{-1}\vert_H,\quad \lambda\in\C^-.
\ee 
The operator $\L$ is called a \textit{selfadjoint dilation} of $A$.
\end{proposition}

\begin{definition} A dilation is \textit{minimal} if it contains no non-trivial reducing part which is itself a selfadjoint dilation of $A$. 
\end{definition}

The minimal  selfadjoint dilation of an arbitrary MDO $A$ will be the sum  of the selfadjoint  part of $A$ and the minimal selfadjoint dilation of the completely non-selfadjoint part. Any completely non-selfadjoint operator has a minimal selfadjoint dilation.

\begin{proposition}[Foia\lfhook{s}]
	The minimal selfadjoint dilation of a completely non-selfadjoint MDO $A$ always has pure absolutely continuous spectrum covering the whole real line, in particular $d\llangle E^\L_\lambda h, h\rrangle$ is an absolutely continuous measure for any $h\in\H$, where $ E^\L_\lambda$ is the spectral resolution of the selfadjoint dilation $\L$.
	\end{proposition}
\begin{corollary}\label{6.1}
Let $A$ be an MDO such that the spectrum of its minimal dilation does not cover the whole real line. Then $A=A^*$.
\end{corollary}

 \section{The Lagrange identity}

Boundary triples are a way of naturally associating `boundary operators' with an adjoint pair of operators. In the abstract setting, Weyl functions can be introduced and many questions e.g.~concerning the extension theory of operators can be investigated in the framework, see e.g.~\cite{BMNW08,DM91} for details. 
 We now discuss  a similar abstract framework for a  maximally dissipative operator and its anti-dissipative adjoint {which allows us to introduce $\Gamma$-operators associated with the imaginary part of the operator $A$. For the case of bounded operators this goes back to the work of the Odessa school on operator knots \cite{Bro71}.}

 \begin{lemma}\label{lemma 3.1}
 Let $A$ be a maximally dissipative operator on a Hilbert space $H$. Then there exists a Hilbert space $E$ and an operator $\Gamma:D(A)\to E$ which is bounded in the graph norm of $A$, has dense range in $E$ and such that for all $u,v\in D(A)$ we have
 \be \label{eq:Lagrange} \llangle Au,v\rrangle_H - \llangle u, Av\rrangle_H = i \llangle \Gamma u,\Gamma v\rrangle_{E}. \ee
 Similarly, there exists a Hilbert space $E_*$ and an operator $\Gamma_*:D(A^*)\to E_*$ which is bounded in the graph norm, has dense range in $E_*$ and such that for all $u,v\in D(A^*)$ we have
 \be \label{eq:LagrangeStar}\llangle A^*u,v\rrangle_H - \llangle u, A^*v\rrangle_H = -i \llangle \Gamma_* u,\Gamma_* v\rrangle_{E_*}.
 \ee
 \end{lemma}
 
\begin{proof}  
Define the sesquilinear form
$$a[u,v]:=\frac{1}{i}\left(\llangle Au,v\rrangle_H - \llangle u, Av\rrangle_H\right) \quad \hbox{ for }\quad u,v\in D(A).$$
Since $A$ is dissipative, $a$ is positive. Moreover, since $A$ is maximal dissipative, $(A+i)^{-1}$ exists and we can define another positive sesquilinear form
$$b[f,g]:=a[(A+i)^{-1}f,(A+i)^{-1}g] \quad \hbox{ for }\quad f,g\in H.$$
Note that for $u\in D(A)$,
$$2\Im\llangle Au,u\rrangle\leq \norm{Au}^2+\norm{u}^2\quad\Rightarrow\quad \Im\llangle Au,u\rrangle\leq \frac14 \left(\norm{Au}^2+2\Im\llangle Au,u\rrangle+\norm{u}^2\right).$$
 Then, with $u=(A+i)^{-1}f$, for the quadratic form we have
\bea
|b[f,f]| &=&  2\left|\Im  \llangle Au,u\rrangle_H \right|\ \leq\ \frac12 \left(\norm{Au}^2+2\Im\llangle Au,u\rrangle+\norm{u}^2\right)\ =\ \frac12\norm{(A+i)u}^2=\frac12\norm{f}^2.
\eea
 Therefore, by the Riesz representation theorem \cite{RS55}, there exists a non-negative bounded operator $F_b:H\to H$ with $\norm{F_b}\leq 1/\sqrt{2}$ such that 
$$b[f,g]=\llangle F_bf,g\rrangle_H=\llangle \sqrt{F_b}f,\sqrt{F_b}g\rrangle_H,\quad f,g\in H $$ 
 and so
 $$ a[u,v]=\llangle \sqrt{F_b}(A+i)u,\sqrt{F_b}(A+i)v\rrangle_H, \quad u,v\in D(A).$$
Let $E=\overline{\Ran(\sqrt{F_b})}=\overline{\Ran(F_b)}$ equipped with the scalar product from $H$ and set $\Gamma=\sqrt{F_b}(A+i):D(A)\to E$. Then $E$ and $\Gamma$ have the required properties and \eqref{eq:Lagrange} holds.

To obtain $E_*$ and $\Gamma_*$, repeat the same construction for the maximal dissipative operator $-A^*$.
\end{proof} 

 \begin{remark}
In general, $E$ and $E_*$ may be of different dimensions, as can be seen in the examples below. However, in the special case of bounded imaginary part of $A$, we can always choose $E=E_*$ and $\Gamma=\Gamma_*$.
\end{remark}

{The next lemma shows that the operator $\Gamma$ is uniquely determined up to unitary transformations.}
\begin{lemma}\label{lem:unique}
{Assume that for any $u,v\in D(A)$ we have
\beq \llangle Au,v\rrangle_H - \llangle u, Av\rrangle_H = i \llangle \Gamma u,\Gamma v\rrangle_{E}= i \llangle \Gammat u,\Gammat v\rrangle_{\Et}, \eeq
where $\Gamma:D(A)\to E$ and $\Gammat:D(A)\to \Et$ are linear maps with dense range in some auxiliary Hilbert spaces $E$ and $\Et$, respectively.}

{
Then there exists a unitary map $V:\Et\to E$ such that $\Gamma u=V\Gammat u$ for all $u\in D(A)$.}
\end{lemma}

\begin{proof}
{ We have that $\llangle \Gamma u,\Gamma v\rrangle_{E}=  \llangle \Gammat u,\Gammat v\rrangle_{\Et}$  for any $u,v\in D(A)$. Assume that $h\in\Ran(\Gamma)$, i.e.~$h=\Gamma u$ for some $u\in D(A)$. Assume that $h=\Gamma u'$ for some other $u'\in D(A)$. Then
\beq
\llangle \Gammat u, \Gammat v \rrangle_\Et \ =\  \llangle \Gamma u, \Gamma v \rrangle_E 
\ =\ \llangle \Gamma u', \Gamma v \rrangle_E 
\ =\ \llangle \Gammat u', \Gammat v \rrangle_\Et.
\eeq
Since $\Ran(\Gammat)$ is dense in $\Et$, we get $\Gammat u=\Gammat u'$. Therefore the map $V:\Ran(\Gamma)\to \Ran(\Gammat)$ given by $h=\Gamma u\mapsto \widetilde{h}=\Gammat u$ is one-to-one from
$\Ran(\Gamma)$ onto $\Ran(\Gammat)$  (by symmetry of $\Gamma$ and $\tilde \Gamma$).
As $\Gamma$ and $\Gammat$ are linear, also $V$ is linear.}

{
Setting $u=v$ we get
$$ \norm{ h}^2=\norm{  \Gamma u}^2=\norm{\Gammat u}^2=   \norm{\widetilde h}^2. 
$$
Thus $V$ is a unitary map from $\Ran(\Gammat)$ onto $\Ran (\Gamma)$.  
Its closure is a unitary operator $V$ from $\Et=\overline{\Ran(\Gammat)}$ onto $E=\overline {\Ran (\Gamma)}$ such that
$$\Gamma u=V\Gammat u, \quad \hbox{ for all } u\in  D(A),$$
as required.
}
\end{proof}

{Despite the lemma formally showing the `uniqueness' of $\Gamma$, its content is purely abstract and of little consequence in applications to particular examples. In most concrete applications, the construction using the square root used in the proof of Lemma \ref{lemma 3.1} does not lead to explicit formulae for the operators $\Gamma$ and $\Gamma_*$. However, in the following we will not make use of this construction of $\Gamma$ and $\Gamma_*$. The theory will instead be valid whenever the identity \eqref{eq:Lagrange} holds and $\Gamma$ and $E$ have the properties stated in the lemma. This is very much in the spirit of the boundary triples approach mentioned above. In our case, we use an abstract Lagrange identity, instead of an abstract Green identity.}

{
Besides the choice of $E,E_*,\Gamma$ and $\Gamma_*$ made in the lemma, this approach allows us the freedom of choosing the operators 
$\Gamma$, $\Gamma_*$ as two versions of the `roots' of the `imaginary part of $A$'. {Even in cases when the roots do not exist, this approach allows us to give meaning to the `roots', and in cases when the roots exist, it enables us to choose an alternative, simpler version of the `root'.} In particular examples, this allows us to choose factorisations \eqref{eq:Lagrange} and \eqref{eq:LagrangeStar} which depend explicitly on parameters of the problem (such as coefficients of a differential expression). Already for the case of a rank two dissipative perturbation of a selfadjoint operator the square root is not explicit, while it is easy to find the `correct' choice of $\Gamma$ in the Lagrange identity. This is illustrated in the following examples.
}

\begin{Example}
\begin{enumerate}
	\item We consider a Schr\"{o}dinger operator with dissipative potential and dissipative boundary condition:
	On $H=L^2(\R^+)$, let $$(Af)(x)=-f''(x)+q(x)f(x),$$
	where $q$ is a measurable and bounded complex-valued function on $\R^+$ with $\Im q(x)\geq 0$ for a.e.~$x\in\R^+$ and
	$$D(A):=\{f\in H^2(\R^+):   f'(0)=hf(0)\}$$
	with $\Im(h)\geq 0$. (The two conditions on the imaginary parts of $q$ and $h$ are necessary and sufficient for $A$ to be maximal dissipative.) Then for $u,v\in D(A)$, we have
	\bea
	\llangle Au, v\rrangle -\llangle u, Av\rrangle &=& u'(0) \overline{v(0)}- u(0) \overline{v'(0)} +2i \int_0^\infty \Im q(x)\ u(x) \overline{v(x)}\ dx \\
	&=& 2i \left( \Im h \ u(0) \overline{v(0)}+\int_0^\infty \Im q(x)\ u(x) \overline{v(x)}\ dx\right).
	\eea
	Let $\Omega=\{x\in\R: \Im q(x)>0\}$, 
	 set $E=\C\oplus L^2(\Omega)$ and 
	$$\Gamma u=\twovec{\sqrt{2\Im h} \ u(0)}{\sqrt{2\Im q}\  u\vert_\Omega}.$$ Then \eqref{eq:Lagrange} holds.
	We remark that in this example $E_*=E$ and $\Gamma_*$ acts in the same way as $\Gamma$, but has a different domain.

\item The next simple example shows that the boundary operators $\Gamma$ and $\Gamma_*$ and the spaces $E$ and $E_*$ can differ significantly. Let 
$$A=i\frac{d}{dx} \quad \hbox{ with }\quad D(A)=H^1_0(\R^+).$$
Then it is easy to check that $A$, being symmetric,  is a  maximally dissipative operator, and
$$A^*=i\frac{d}{dx} \quad \hbox{ with }\quad D(A^*)=H^1(\R^+),$$
is an anti-dissipative operator, and we can choose
$\Gamma=0$ with $E=\{0\}$ and $\Gamma_*u=u(0)$ with $E_*=\C$.
\end{enumerate}
\end{Example}

We conclude this section with two useful identities which follow from the Lagrange identity. 

\begin{lem}\label{lemma:2} (Abstract Green Function Identities) For $\lambda\in\C^+$ and $\mu\in\C^-$ we have
\be (A+\lambda)^{-1}-(A^*+\mu)^{-1} +(\lambda-\mu)(A^*+\mu)^{-1}(A+\lambda)^{-1} =  
 -i(\Gamma(A+\overline{\mu})^{-1})^*(\Gamma(A+\lambda)^{-1}) \label{Green} \ee
and
\be
(A+\lambda)^{-1}-(A^*+\mu)^{-1} +(\lambda-\mu)(A+\lambda)^{-1} (A^*+\mu)^{-1}=  
 -i(\Gamma_*(A^*+\overline{\lambda})^{-1})^*(\Gamma_*(A^*+\mu)^{-1}) 
\label{GreenStar}
\ee

\end{lem}

\begin{proof} The first result is equivalent to 
\bea \llangle (A+\lambda)^{-1}f,g\rangle - \langle (A^*+\mu)^{-1}f,g\rrangle  &=& 
 -i \langle (\Gamma(A+\lambda)^{-1})f,(\Gamma(A+\overline{\mu})^{-1})g\rangle \\
 && - (\lambda-\mu)\langle
 (A+\lambda)^{-1}f,(A+\overline{\mu})^{-1}g\rangle 
\eea
for all $f,g\in H$.
Defining $f_\lambda = (A+\lambda)^{-1}f$, $g_\mu = (A+\overline{\mu})^{-1}g$, this in turn is equivalent to
\[ \langle f_\lambda,(A+\overline{\mu})g_\mu\rangle - \langle (A+\lambda)f_\lambda,g_\mu\rangle
 = -i\langle \Gamma f_\lambda, \Gamma g_\mu\rangle - (\lambda-\mu)\langle f_\lambda, g_\mu\rangle \hbox{ for all } f_\lambda,g_\mu\in D(A). \]
This is precisely the Lagrange identity \eqref{eq:Lagrange}. 
The proof of \eqref{GreenStar} is similar.
\end{proof}

\begin{remark}\label{rem:gen}
The identities \eqref{Green} and \eqref{GreenStar} clearly extend to all $\lambda,\mu$ such that $-\lambda,-\mub\in\rho(A)$.
\end{remark}

\section{The \v{S}traus characteristic function and its properties}\label{Straus}

The first characteristic function, discussed below, was introduced by Liv\v{s}ic \cite{Liv73}. Later, by completely different methods, a characteristic function was introduced by Sz.-Nagy and Foia\lfhook{s} \cite{SFBK10} as part of their harmonic analysis of contractions. As was clarified by M.~Krein and Gohberg, the Sz-Nagy-Foia\lfhook{s} charactersitic function is a generalisation of the 
 Liv\v{s}ic characteristic function to a wider class of operators. Simultaneously, in a series of papers by \v{S}traus \cite{Str60,Str68}, another (unitarily equivalent) characteristic function was introduced in his study of extensions of symmetric operators and also in more general settings. We will introduce the \v{S}traus characteristic function in our setting and discuss its connection to the Sz-Nagy-Foia\lfhook{s} charactersitic function below. {This definition of the characteristic function goes back to the idea of the characteristic function of an operator knot as introduced by the Odessa school \cite{Bro71}. It is also related to the characteristic functions in the setting of boundary triples, introduced by Derkach and Malamud, see, e.g.~\cite{DM92,DM95}.}

We recall that in all of the following $A$ is a maximally dissipative operator on $H$ and $\Gamma,\Gamma_*$ and $E$, $E_*$ are operators and, respectively, spaces with the properties given in Lemma \ref{lemma 3.1}. 
We start with a simple identity.

\begin{lem}\label{lemma:2.1} 
For all $u\in D(A)$ and $z\in \rho(A^*)$ we have
\[ \| \Gamma_*(A^*-z)^{-1}(A-z) u \|^2 = 
 \| \Gamma u \|^2 - 2\Im(z)\| (A^*-z)^{-1}(A-z)u - u \|^2. \]
\end{lem}
\begin{proof}
This is an explicit calculation. For $u\in D(A)$ we have, by the second Lagrange identity \eqref{eq:LagrangeStar},
\be
\begin{array}{lcl} & &  \hspace{-1.2cm}i\langle \Gamma_* (A^*-z)^{-1}(A-z)u,\Gamma_* (A^*-z)^{-1}(A-z)u\rangle_{E_*} \\
 & & \\
& = & -\langle A^*(A^*-z)^{-1}(A-z)u,(A^*-z)^{-1}(A-z)u\rangle \\
 & & + \langle (A^*-z)^{-1}(A-z)u,A^*(A^*-z)^{-1}(A-z)u\rangle \\
 & & \\
&  =  & -\langle (A-z)u,(A^*-z)^{-1}(A-z)u\rangle + \langle (A^*-z)^{-1}(A-z)u,(A-z)u\rangle \\
& &  -2i\Im(z)\| (A^*-z)^{-1}(A-z)u \|^2 \\ & & \\
& = & -\langle u, (A^*-\overline{z})(A^*-z)^{-1}(A-z)u\rangle 
   + \langle (A^*-\overline{z})(A^*-z)^{-1}(A-z)u,u\rangle \\
 & &   -2i\Im(z) \| (A^*-z)^{-1}(A-z)u \|^2 \\ & & \\
& = & -\langle u,(A-z)u\rangle + \langle (A-z)u,u\rangle + 2i\Im(z)\langle u,(A^*-z)^{-1}(A-z)u\rangle \\
 & &  + 2i\Im(z)\langle (A^*-z)^{-1}(A-z)u,u\rangle -2i\Im(z)\| (A^*-z)^{-1}(A-z)u \|^2. \end{array}
\ee
By the first Lagrange identity \eqref{eq:Lagrange}, this is equal to
\[ i \| \Gamma u \|^2  -2i\Im(z)\| u \|^2   +4i\Im(z)\Re\langle u,(A^*-z)^{-1}(A-z)u\rangle
 -2i\Im(z)\| (A^*-z)^{-1}(A-z)u \|^2,
\]
which, in turn, simplifies to
\[ i\left[ \| \Gamma u \|^2 -2\Im(z) \| (A^*-z)^{-1}(A-z)u-u \|^2\right], \]
proving the identity.
\end{proof}

\begin{cor} \label{corollary:2.2} For $u\in D(A)$ and $z\in \Cc^+$,
\be \| \Gamma_* (A^*-z)^{-1}(A-z)u\| \leq \| \Gamma u \|. \label{eq:contraction}\ee
Hence, there exists a unique contraction $S(z):E\to E_*$, analytic in the upper half-plane, such that
\be S(z)\Gamma u = \Gamma_*(A^*-z)^{-1}(A-z)u \hbox{ for all } u\in D(A).\label{eq:S} \ee
\end{cor}

\begin{proof}
Define $S(z)$ on $\Ran(\Gamma)$ by \eqref{eq:S}. Then $S(z)$ is both well-defined and contractive by \eqref{eq:contraction}. Therefore, it can be uniquely extended to a contraction on $E$. Analyticity follows from analyticity of the right hand side of \eqref{eq:S}. 
\end{proof} 

\begin{lem}\label{lemma:2.1b} For $u\in D(A^*)$ and $z\in\rho(A)$,
\[ \| \Gamma (A-z)^{-1}(A^*-z)u\|^2 = \|\Gamma_* u \|^2 + 2\Im(z)\| (A-z)^{-1}(A^*-z)u + u \|^2. \]
Correspondingly, for $z\in\C^-$ there exists a contraction $S_*(z):E_*\to E$, analytic in the lower half-plane, such that
\be S_*(z)\Gamma_* u = \Gamma (A-z)^{-1}(A^*-z)u. \label{Sstar} \ee
\end{lem}

\begin{proof}
This is analogous to the proof of Lemma \ref{lemma:2.1} and Corollary \ref{corollary:2.2}.
\end{proof}

We now wish to extend $S(z)$ by \eqref{eq:S} to all $z\in\rho(A^*)$ and  $S_*(z)$ by \eqref{Sstar} to all $z\in\rho(A)$.

\begin{lemma}
$S(z)$ is well-defined on $\Ran\Gamma$ for $z\in\rho(A^*)$ and $S_*(z)$ is well-defined on $\Ran\Gamma_*$ for $z\in\rho(A)$.
\end{lemma}

\begin{proof}
We prove the result for $S(z)$, the proof for $S_*(z)$ is similar. We  need to show that if $u\in D(A)$ with $\Gamma u=0$, then for any $z\in\rho(A^*)$ we have $\Gamma_* (A^*-z)^{-1}(A-z)u=0$. If $\Gamma u=0$ then by \eqref{eq:Lagrange}, we have that $\llangle Au,v\rrangle=\llangle u, Av\rrangle$ for any $v\in D(A)$. This implies that $u\in D(A^*)$ with $A^*u=Au$, so $\Gamma_* (A^*-z)^{-1}(A-z)u=\Gamma_* u$. Using \eqref{eq:LagrangeStar} and again \eqref{eq:Lagrange}, we get that 
$$-i\norm{\Gamma_* u}^2=\llangle A^*u,u\rrangle - \llangle u,A^*u\rrangle=\llangle Au,u\rrangle - \llangle u,Au\rrangle = i \norm{\Gamma u}^2=0,$$
as required.
\end{proof}

The following lemma gives a useful identity for the difference of $S$ at two different points. 
\begin{lemma}
For $\mu,\mut\in\rho(A^*)$, we have the following identity:
\be\label{eq:Sdiff} S(\mu)-S(\mut)= i(\mu-\mut) \left(\Gamma_* (A^*-\mu)^{-1}\right)\left(\Gamma(A-\mutb)^{-1}\right)^* \hbox{ on } \Ran\Gamma.\ee
\end{lemma}

\begin{proof}
 Let $h\in D(A)$ and set
$$v=\left[ S(\mu)-S(\mut)- i(\mu-\mut) \left(\Gamma_* (A^*-\mu)^{-1}\right)\left(\Gamma(A-\mutb)^{-1}\right)^*\right]\Gamma h.$$
Then using \eqref{eq:S} we get
\bea
v&=&  \Gamma_* (A^*-\mu)^{-1}(A-\mu) h - \Gamma_* (A^*-\mut)^{-1}(A-\mut) h+i(\mut-\mu)\left(\Gamma_* (A^*-\mu)^{-1}\right)\left(\Gamma(A-\mutb)^{-1}\right)^*\Gamma h\\
&=& \Gamma_* (A^*-\mu)^{-1}
\left[ (A-\mu)h - (A-\mut) h - (\mut-\mu) (A^*-\mut)^{-1}(A-\mut) h
+i(\mut-\mu)\left(\Gamma(A-\mutb)^{-1}\right)^*\Gamma h\right]
\\
&=& (\mut-\mu)\Gamma_* (A^*-\mu)^{-1}
\left[  h -  (A^*-\mut)^{-1}(A-\mut) h
+i\left(\Gamma(A-\mutb)^{-1}\right)^*\Gamma h\right]
\eea
Next, set $h=(A^*-\mutb)^{-1}g$. After applying \eqref{Green} (with Remark \ref{rem:gen}), a short calculation shows that the term in the square brackets vanishes, giving $v=0$, as required.
\end{proof}

From this identity we see that although $S(\mu)$ need not be a contraction for all $\mu\in\rho(A^*)$, it remains bounded on $\Ran\Gamma$.

\begin{corollary}
For $\mu\in\rho(A^*)$, we have that $S(\mu)$ is a bounded operator on $\Ran\Gamma$. Moreover,
\be\label{eq:normestimate}
\norm{S(\mu)}\leq 1+16\sqrt{\gamma(1+\gamma)},
\ee 
where $\gamma=|\Im(\mu)|\norm{(A^*-\mu)^{-1}}$.
\end{corollary}

\begin{proof}
Choose $\mut\in\C^+$. Then from \eqref{eq:Sdiff}, we get that
\bea
\norm{S(\mu)}&\leq& \norm{S(\mut)} + |\mu-\mut| \norm{\Gamma_* (A^*-\mu)^{-1}}\norm{\Gamma(A-\mutb)^{-1}}\\
&\leq &1 +|\mu-\mut| \norm{\Gamma_* (A^*-\mu)^{-1}}\norm{\Gamma(A-\mutb)^{-1}}.
\eea
Let $A'=A-\Re(\mut)$. Then, using the Lagrange identity \eqref{eq:Lagrange},  
\bea
\norm{\Gamma(A-\mutb)^{-1}}^2 &\leq& 2 \norm{(A-\mutb)^{-1}} \norm{A'(A-\mutb)^{-1}}\\
&\leq & \frac{2}{\Im(\mut)}\norm{I+i\Im(\mutb)(A-\mutb)^{-1}}\ \leq \ \frac{4}{\Im(\mut)}.
\eea
Next, let $\lambda=\Re(\mu)+i\tau$ form some $\tau>0$. Then using the previous estimate, we get
\bea
\norm{\Gamma_* (A^*-\mu)^{-1}} &\leq & \norm{\Gamma_* (A^*-\lambda)^{-1}\left[I+(\mu-\lambda)(A^*-\mu)^{-1}\right]}\\
&\leq & \norm{\Gamma_* (A^*-\lambda)^{-1}} \norm{I+i(\Im(\mu)-\tau)(A^*-\mu)^{-1}} \\
&\leq &  \sqrt{\frac{4}{\tau}}\left[1+(|\Im(\mu)|+\tau) \norm{(A^*-\mu)^{-1}}\right].
\eea
Combining the estimates, we get that
\beq
\norm{S(\mu)} \leq 1+\frac{4(|\Im(\mu)|+\Im(\mut))}{\sqrt{\Im{\mut}}} \cdot \frac{1+(|\Im(\mu)|+\tau) \norm{(A^*-\mu)^{-1}}}{\sqrt{\tau}}.
\eeq
Both fractions are of the form $(a+bx)/\sqrt{x}$ and are minimized for $x=a/b$ with value $2\sqrt{ab}$. Thus,
\beq
\norm{S(\mu)} \leq 1+16 \sqrt{\Im{\mu}} \cdot \sqrt{\left(1+|\Im(\mu)| \norm{(A^*-\mu)^{-1}}\right)\norm{(A^*-\mu)^{-1}} },
\eeq
as claimed.
\end{proof}

This justifies the following definition.
\begin{definition}
The operator-valued function $S(\cdot)$, defined for $z\in\rho(A^*)$  by \eqref{eq:S} on $\Ran(\Gamma)$ and extended to $E$ by continuity is called the \v{S}traus characteristic function of the operator $A$.
\end{definition}

\begin{lemma}\label{lemma:adj}
$S(z)=S_*^*(\zb)$ for $z	\in \rho(A^*)$.
\end{lemma} 

\begin{proof}
We need to show that for all $e\in E$ and $e_*\in E_*$, we have
$$\llangle S(z) e, e_*\rrangle_{E_*} =\llangle e, S_*(\zb)e_*\rrangle_E.$$
Due to the boundedness of $S(z)$ and $S_*(\zb)$, it is sufficient to show this on dense sets. Therefore, we choose $e=\Gamma(A-\zb)^{-1}f$ and $e_*=\Gamma_*(A^*-z)^{-1}g$ for $f,g\in H$.
Then
\bea
 i \llangle S(z) e, e_*\rrangle_{E_*}&=& 
 i \llangle \Gamma_*(A^*-z)^{-1}(A-z)(A-\zb)^{-1}f, \Gamma_*(A^*-z)^{-1}g\rrangle_{E_*}
\eea
and using the Lagrange identity \eqref{eq:LagrangeStar} this equals
\bea  & -\llangle (A^*-\zb)(A^*-z)^{-1}(A-z)(A-\zb)^{-1}f, (A^*-z)^{-1}g \rrangle + \llangle(A^*-z)^{-1}(A-z)(A-\zb)^{-1}f, g \rrangle
\\ &=
\llangle \left(I-(A-\zb)^{-1}(A^*-\zb)\right)(A^*-z)^{-1}(A-z)(A-\zb)^{-1}f, g \rrangle\\
&=\llangle \left((A^*-z)^{-1}- (A-\zb)^{-1}\left(I+(z-\zb)(A^*-z)^{-1}\right)\right)(A-z)(A-\zb)^{-1}f, g \rrangle\\
&=\llangle \left((A^*-z)^{-1}- (A-\zb)^{-1} -(z-\zb)(A-\zb)^{-1} (A^*-z)^{-1}\right)(A-z)(A-\zb)^{-1}f, g \rrangle
.\eea
On the other hand, using the first Lagrange identity \eqref{eq:Lagrange} we have 
\bea
i\llangle e, S_*(\zb)e_*\rrangle_E &=&
i\llangle \Gamma(A-\zb)^{-1}f, \Gamma (A-\zb)^{-1}(A^*-\zb)(A^*-z)^{-1}g\rrangle\\ 
&=& \llangle (A-z)(A-\zb)^{-1}f, (A-\zb)^{-1}(A^*-\zb)(A^*-z)^{-1}g \rrangle -\llangle (A-\zb)^{-1}f, (A^*-\zb)(A^*-z)^{-1}g\rrangle\\
&=& \llangle (A-z)(A-\zb)^{-1}f, ((A-\zb)^{-1}(A^*-\zb)-I)(A^*-z)^{-1}g \rrangle \\
&=& \llangle (A-z)(A-\zb)^{-1}f, \left((A-\zb)^{-1}(I+(z-\zb)(A^*-z)^{-1})-(A^*-z)^{-1}\right)g \rrangle\\
&=& \llangle (A-z)(A-\zb)^{-1}f, \left((A-\zb)^{-1}-(A^*-z)^{-1}+(z-\zb)(A-\zb)^{-1}(A^*-z)^{-1}\right)g \rrangle\\
&=& \llangle \left((A^*-z)^{-1}- (A-\zb)^{-1} -(z-\zb)(A-\zb)^{-1} (A^*-z)^{-1}\right)(A-z)(A-\zb)^{-1}f, g \rrangle
,
\eea
proving the desired equality.
\end{proof}

\begin{lemma}\label{lem:inverse}
$S(z)S_*(z)=I_{E_*}$ and $S_*(z)S(z)=I_E$ whenever  $z	\in \rho(A)\cap\rho(A^*)$.
\end{lemma} 

\begin{proof}
Due to the boundedness of the operators involved, it is again sufficient to show this on a dense set. Let $u\in D(A^*)$. Then
\bea
S(z)S_*(z)\Gamma_*u &=& S(z)\Gamma (A-z)^{-1}(A^*-z)u \\
&=& \Gamma_*(A^*-z)^{-1}(A-z)(A-z)^{-1}(A^*-z)u \ =\ \Gamma_* u.
\eea
The second equality can be proved similarly.
\end{proof} 
 
This immediately gives the following results.
\begin{corollary}
\begin{enumerate}
	\item $S(z)$ is unitary for $z\in\R\cap\rho(A)$.
\item If $\sigma(A)$ does not cover the whole upper half plane (or if $\rho(A)\cap\rho(A^*)\not=\emptyset$), then $\dim E=\dim E_*$.
\end{enumerate}
\end{corollary}

\begin{proof}
\begin{enumerate}
	\item This is immediate from Lemmas \ref{lemma:adj} and \ref{lem:inverse}.
	\item Let  $z\in\rho(A)\cap\rho(A^*)$. From the first equality in Lemma \ref{lem:inverse}, we get that 
	$$\dim E_*\leq \rank S(z)\leq\min\{\dim E,\dim E_*\},$$
	while the second equality gives
	$$\dim E\leq \rank S_*(z)\leq\min\{\dim E,\dim E_*\},$$
	proving the result.
\end{enumerate}
\end{proof}

The next lemma shows Hermitian positivity properties (see Azizov-Iokhvidov \cite{AI89} for related results).

\begin{lemma}\label{lem:SStarS}
For  $w,z	\in \C^+$, we have
$$\frac{1}{\bar{w}-z}\left(I_E-S^*(w)S(z)\right)=i \left(\Gamma(A-\bar{w})^{-1}\right) \left(\Gamma(A-\bar{z})^{-1}\right)^*$$
and for $w,z	\in \C^-$, we have
$$\frac{1}{\bar{w}-z}\left(I_{E_*}-S_*^*(w)S_*(z)\right)= - i \left(\Gamma_*(A^*-\bar{w})^{-1}\right) \left(\Gamma_*(A^*-\bar{z})^{-1}\right)^*.$$
\end{lemma}

\begin{proof}
We check the first equality on a dense set. Let $u=\Gamma(A-\bar{z})^{-1}f$ for some $f\in H$. Then using that $S^*(w)=S_*(\bar{w})$ from Lemma \ref{lemma:adj}, we get
\bea
\left(I_E-S^*(w)S(z)\right) u 
&= & \left[\Gamma(A-\bar{z})^{-1} - S^*(w)\Gamma_*(A^*-z)^{-1}(A-z)(A-\bar{z})^{-1}\right]f\\
&=& \left[\Gamma(A-\bar{z})^{-1} - \Gamma (A-\bar{w})^{-1}(A^*-\bar{w})(A^*-z)^{-1}(A-z)(A-\bar{z})^{-1}\right]f\\
&=& \Gamma (A-\bar{w})^{-1}\left[ I + (\bar{z}-\bar{w})(A-\bar{z})^{-1} - (I +(z-\bar{w})(A^*-z)^{-1})  (I+(\bar{z}-z)(A-\bar{z})^{-1})\right]f\\
&=& (z-\bar{w})\Gamma (A-\bar{w})^{-1}\left[  (A-\bar{z})^{-1} - (A^*-z)^{-1} -(A^*-z)^{-1}  (\bar{z}-z)(A-\bar{z})^{-1}\right]f.
\eea
Using \eqref{Green} with $\lambda=-\zb$ and $\mu=-z$, we get 
\bea
\left(I_E-S^*(w)S(z)\right) u &=& (z-\bar{w})\Gamma (A-\bar{w})^{-1}(-i)(\Gamma(A-\zb)^{-1})^*(\Gamma(A-\zb)^{-1}) f\\
&=& i (\bar{w}-z) \left(\Gamma(A-\bar{w})^{-1}\right) \left(\Gamma(A-\bar{z})^{-1}\right)^* u,
\eea
proving the identity.
\end{proof}

\begin{remark}
In the case when $z=w$, the rank of the limit operator of $I_E-S^*(z)S(z)$ as $z$ tends to the real axis corresponds to the local multiplicity of the a.c.-spectrum of $A$ {(see \cite{Pav75,Tik95}).}
\end{remark}

For later calculations, we will also need the following identities:

\begin{lem}
For any $u\in H$, $\mu,z\in\C^-$ we have
\be\label{eq:Sgammastar}
\left(\Gamma_*(A^*-\bar{\mu})^{-1}\right)^*S(\bar{z}) = \left[I-(\bar{z}-\mu)(A-\mu)^{-1}\right]\left(\Gamma(A-z)^{-1}\right)^*
\ee
and 
\be\label{eq:Stildegammastar}
\left(\Gamma(A-\mu)^{-1}\right)^*S_*(z) = \left[I-(z-\bar{\mu})(A^*-\bar{\mu})^{-1}\right]\left(\Gamma_*(A^*-\bar{z})^{-1}\right)^*.
\ee
\end{lem}

\begin{proof}
From the definition of $S_*$ in \eqref{Sstar} we have that
$$S_*(z)\Gamma_* (A^*-\bar{\mu})^{-1} = \Gamma (A-z)^{-1}(A^*-z) (A^*-\bar{\mu})^{-1}.$$
Taking adjoints, using Lemma \ref{lemma:adj}, we get
\bea\left(\Gamma_* (A^*-\bar{\mu})^{-1} \right)^* S(\bar{z})&=& \left[(A^*-z) (A^*-\bar{\mu})^{-1}\right]^* \left(\Gamma (A-z)^{-1}\right)^*\\
&=&  \left[I -(z-\bar{\mu})(A^*-\bar{\mu})^{-1}\right]^* \left(\Gamma (A-z)^{-1}\right)^* \\
&=&  \left[I-(\bar{z}-\mu)(A-\mu)^{-1}\right]\left(\Gamma(A-z)^{-1}\right)^*.
\eea
This proves \eqref{eq:Sgammastar}. The proof of \eqref{eq:Stildegammastar} is similar.
\end{proof}

\begin{example}\label{ex:bcq}
We consider a Schr\"odinger operator on the half-line with a dissipative boundary condition and potential. Let 
$H=L^2(\R^+)$ and \be(Af)(x)=-f''(x)+q(x)f(x),\label{eq:Schr}\ee
	where $q$ is a measurable and bounded function on $\R^+$ with $\Im q\geq 0$ and the domain of the operator is given by
	\be\label{eq:Schrdom}D(A):=\{y\in H^2(\R^+): y'(0)=hy(0)
	\},\ee
	where $\Im h >0$. 
		Then for $u,v\in D(A)$, equation \eqref{eq:Lagrange} holds with 
	$$\Gamma u=\twovec{\sqrt{2\Im h}\ u(0)}{\sqrt{2\Im q}\ u},$$
which has dense range in $\C\oplus L^2(\{x\in \R^+: \Im q(x)>0\})$. 
Note that $A^*$ is given by  
	$$(A^*f)(x)=-f''(x)+\overline{q(x)}f(x) \quad \hbox{ with domain }\quad D(A^*):=\{y\in H^2(\R^+): y'(0)=\bar{h}y(0)\}.$$
	 $\Gamma_*$ is given by the same expression as $\Gamma$.
	
	We now calculate the characteristic function.
For $u\in D(A)$ consider $S(z)\Gamma u= \Gamma_*(A^*-z)^{-1}(A-z)u$ with $z\in\rho(A^*)$.
		We choose to $\varphi_*$ and $\psi_*$ to be the fundamental solutions of $-y''+\bar{q}y=\lambda y$ with $\varphi_*$ and $\psi_*$ normalized by
			\be\label{eq:norm} \twovec{\varphi_*(0)}{\varphi_*'(0)}=\twovec{0}{1}\quad\hbox{ and }\quad \twovec{\psi_*(0)}{\psi_*'(0)}=\twovec{1}{0}.\ee
		 Moreover, let $m_*$ denote the Weyl-Titchmarsh function associated with $-y''+\bar{q}y$, i.e.~$m_*(z)\varphi_*+\psi_*$ is the $L^2$-solution to $-y''+\bar{q}y=zy$. This solution is unique up to constants due to our assumptions on $q$.
 Choosing 
\be\label{eq:cstar} c_*(z)=\frac{ h-\bar{h} }{\bar{h}-m_*(z)}u(0),\ee
 we have that
$v=u+c_*(z)(m_*(z)\varphi_*+\psi_*)+2i(A^*-z)^{-1} (\Im q)\ u$ lies in $D(A^*)$ and
$$(A^*-z) v = -u''+(\bar{q}-z)u +2i (\Im q)\ u = (A-z)u.$$
Hence, $v=(A^*-z)^{-1}(A-z)u$ and for $z\in\rho(A^*)$ we have $S(z)\Gamma u= \Gamma_*v$, i.e.~
$$S(z)\twovec{\sqrt{2\Im h}\ u(0)}{\sqrt{2\Im q}\ u} = \twovec{\sqrt{2\Im h}\ v(0)}{\sqrt{2\Im q}\ v}=\twovec{\sqrt{2\Im h}\ \left( u(0)+c_*(z)+\left(2i(A^*-z)^{-1} (\Im q)\ u\right)(0)\right)}{\sqrt{2\Im q}\ \left(u+c_*(z)(m_*(z)\varphi_*+\psi_*)+2i(A^*-z)^{-1} (\Im q)\ u\right)}.$$
This implies that
\ben
S(z)&=&\twomat{I+\frac{h-\bar{h}}{\bar{h}-m_*(z)}}{i\sqrt{2\Im h}\ (A^*-z)^{-1} \sqrt{2\Im q}\ \cdot (0)}
{i\frac{\sqrt{2\Im h}\ \sqrt{2\Im q}\ }{\bar{h}-m_*(z)}(m_*(z)\varphi_*(x)+\psi_*(x))}
{I+i \sqrt{2\Im q}\ (A^*-z)^{-1} \sqrt{2\Im q}}\nonumber \\
&=& \twomat{\frac{h-m_*(z)}{\bar{h}-m_*(z)}}{i\sqrt{2\Im h}\ \int_0^\infty \frac{m_*(z)\varphi_*(y)+\psi_*(y)}{\bar{h}-m_*(z)}\sqrt{2\Im q(y)}\ \cdot (y) \ dy}
{i\frac{\sqrt{2\Im h}\ \sqrt{2\Im q}\ }{\bar{h}-m_*(z)}(m_*(z)\varphi_*(x)+\psi_*(x))}
{I+i \sqrt{2\Im q}\ (A^*-z)^{-1} \sqrt{2\Im q}}. \label{SchrCharFn}
\een	
Note that the top left entry is the same as  the well-known formula of Pavlov
 for the case of real $q$, where the functions $m$ and $m_*$ coincide:
\be\label{Pavlov}S(z) =\left(1+\frac{h-\bar{h}}{\bar{h}-m(z)}\right).\ee 
Pavlov deduced it from the scattering theory interpretation of the characteristic function {\cite{Pav74,Pav76}. }
The bottom right entry agrees with the Liv\v{s}ic characteristic function for the case with a selfadjoint boundary condition \cite{Liv73}. Moreover, this formula shows the connection between the Weyl $m$-function and the characteristic function for this example.
\end{example}

We conclude this section by showing that the \v{S}traus characteristic function as defined here coincides up to an isometric transformation with the Sz-Nagy-Foias characteristic function for a contraction $T$, see \cite{SFBK10}, given by
\be\label{SzN}
\Theta(\lambda)=\left[-T +\lambda D_{T^*}(I-\lambda T^*)^{-1}D_T \right]\vert_{\Ran{\overline{D_T}}}, \quad\hbox{for } |\lambda|<1,
\ee
where $D_T=\sqrt{I-T^*T}$ and $D_{T^*}=\sqrt{I-TT^*}$. Then $\Theta(\lambda)\in B(\overline{\Ran(D_T)},\overline{\Ran(D_{T^*})})$. 

\begin{proposition} Let $z\in\C^+$ and $A$ be a maximal dissipative operator with Cayley transform $T$ given by \eqref{Cayley}. Then  
there exist isometric surjective operators $U:E\to \overline{\Ran(D_T)}$ and $U_*:E_*\to\overline{\Ran(D_{T^*})}$ such that
$$\Theta(\lambda)=U_*S(z)U^*, \hbox{ where } \lambda=\frac{z-i}{z+i}.$$
\end{proposition}

\begin{proof}
We first determine some of the expressions arising in $\Theta(\lambda)$ in terms of $A$ and $A^*$. We have
\ben\label{DT}
D_T^2&=& I-T^*T \ =\ I-(I+2i(A^*-i)^{-1})(I-2i(A+i)^{-1})\\ \nonumber
&=& 2i\left[(A+i)^{-1}-(A^*-i)^{-1}+2i (A^*-i)^{-1}(A+i)^{-1}  \right].
\een
Similarly,
\ben\label{DTstar}
D_{T^*}^2&=& 2i\left[(A+i)^{-1}-(A^*-i)^{-1}+2i (A+i)^{-1}(A^*-i)^{-1} \right].
\een
Moreover,
\bea
I-\lambda T^* &=& I-\lambda-2i\lambda(A^*-i)^{-1} \ =\ \left((1-\lambda)(A^*-i)-2i\lambda\right) (A^*-i)^{-1}\\
&=& (1-\lambda)\left(A^*-i-\frac{2i\lambda}{1-\lambda}\right)(A^*-i)^{-1} \ =\  (1-\lambda)\left(A^*-z\right)(A^*-i)^{-1},
\eea
so 
\ben\label{res}
(I-\lambda T^*)^{-1} &=&  \frac{1}{1-\lambda}(A^*-i)\left(A^*-z\right)^{-1}.
\een
From \eqref{eq:Lagrange}, we have that for any $f,g\in H$,
\bea
 i \llangle \Gamma (A+i)^{-1}f,\Gamma (A+i)^{-1}g\rrangle_{E}&=&
\llangle A(A+i)^{-1}f,(A+i)^{-1}g\rrangle_H - \llangle (A+i)^{-1}f, A(A+i)^{-1}g\rrangle_H \\
&=& \llangle \left[(A^*-i)^{-1}A(A+i)^{-1}-A^*(A^*-i)^{-1}(A+i)^{-1}\right]f,g\rrangle_H\\
&=& \llangle \left[(A^*-i)^{-1} -(A+i)^{-1}    -2i (A^*-i)^{-1}(A+i)^{-1}\right]f,g\rrangle_H\\
&=& \llangle \frac{i}{2}D_T^2f,g\rrangle_H.
\eea
This implies that  $\frac{1}{2}D_T^2=\left(\Gamma(A+i)^{-1}\right)^* \Gamma (A+i)^{-1}.$ Similarly,
$\frac{1}{2}D_{T^*}^2=\left(\Gamma_*(A^*-i)^{-1}\right)^* \Gamma_* (A^*-i)^{-1}.$
Therefore there exist $U$ and $U_*$ with the desired properties, such that 
$$D_T=\sqrt{2}\ U\Gamma (A+i)^{-1} \quad\hbox{ and }\quad D_{T^*}=-\sqrt{2}\ U_*\Gamma_*(A^*-i)^{-1}.$$
Noting that $TD_T^2=D_{T^*}^2T$ and thus, using the functional calculus $TD_T=D_{T^*}T$, we now have with \eqref{DT} and \eqref{res} that
\bea
\Theta(\lambda)D_T&=& -D_{T^*}T+\frac{2i\lambda}{1-\lambda}D_{T^*}(A^*-i)\left(A^*-z\right)^{-1}\left[(A+i)^{-1}-(A^*-i)^{-1}+2i (A^*-i)^{-1}(A+i)^{-1}  \right]\\
&=& D_{T^*}\left(-I+2i(A+i)^{-1}+(z-i)(A^*-i)\left(A^*-z\right)^{-1}\left[(A+i)^{-1}-(A^*-i)^{-1}+2i (A^*-i)^{-1}(A+i)^{-1}  \right]\right)\\
&=& D_{T^*}(A^*-i)\left(A^*-z\right)^{-1}\\
&& \quad
\left[(I-(z-i)(A^*-i)^{-1})(-I+2i(A+i)^{-1})
+(z-i)\left[(A+i)^{-1}-(A^*-i)^{-1}+2i (A^*-i)^{-1}(A+i)^{-1}  \right]\right]\\
&=& D_{T^*}(A^*-i)\left(A^*-z\right)^{-1}\left(-I+(z+i)(A+i)^{-1}\right).
\eea
Using our expressions for $D_{T^*}$ and $D_T$, we get
\bea
\Theta(\lambda)D_T&=& \sqrt{2}\ U_*\Gamma_*\left(A^*-z\right)^{-1}\left(-I+(z+i)(A+i)^{-1}\right)\\
&=& \sqrt{2}\ U_*\Gamma_*\left(A^*-z\right)^{-1}(A-z)(A+i)^{-1}\\
&=& \sqrt{2}\ U_*S(z)\Gamma(A+i)^{-1}\ = \ U_*S(z)U^*D_T,
\eea
proving that $\Theta(\lambda)=U_*S(z)U^*.$
\end{proof}

 \begin{remark}
This shows that the two approaches differ in the choice of the root of the imaginary part of the operator. The advantage of the \v{S}traus characteristic function for us is that we can often explicitly determine $\Gamma$ and $\Gamma_*$, while it is rarely possible to find explicit expressions for $D_T$ and $D_{T^*}$. 
\end{remark}

\section{Definition of the dilation}
 
Before studying the dilation itself, we first introduce its domain and show that it has several equivalent descriptions. Here and in what follows we let $\cH=L^2(\R^-,E_*)\oplus H \oplus L^2(\R^+,E)$, where $L^2(\R^-,E_*)$ and $L^2(\R^+,E)$ are suitable channels in the sense of Lax and Phillips \cite{LP67}.

\begin{definition} Let $\mu\in\C^-$ and $\lambda\in\C^+$. We define the subset $D(\cL)$ of $\cH$ by
\begin{eqnarray}\label{def:L}
 D(\cL) & = & \left\{ U=\U \; : \; u\in H,\; v_+\in H^1(\R^+,E),\; v_-\in H^1(\R^-,E_*), \; \right. \\
 & & \;\;\; u+(\Gamma_*(A^*+\mu)^{-1})^*v_{-}(0)\in D(A),\; u+(\Gamma(A+\lambda)^{-1})^*v_{+}(0)\in D(A^*), \nonumber \\
  & & \;\;\; \mbox{(I)} \; v_+(0)=S^*(-\mu)v_-(0)+i\Gamma\left(u+(\Gamma_*(A^*+\mu)^{-1})^*v_{-}(0)\right),
\nonumber \\
  & &    \;\;\; \mbox{(II)} \; v_-(0)=S(-\bar{\lambda})v_+(0)-i\Gamma_*\left( u+(\Gamma(A+\lambda)^{-1})^*v_{+}(0)\right)\quad\quad\quad\quad\quad\Bigg\} . \nonumber 
\end{eqnarray}
\end{definition} 
We will see in Lemma \ref{lem:5.4} that the conditions $(I)$ and $(II)$ are equivalent, so it is possible to omit one of them in the definition.
The numbers $\mu$ and $\lambda$ are regularisation parameters. For special situations, such as if the imaginary part of $A$ is relatively bounded, they are not needed. For general MDOs, however, the regularisation is necessary. Besides the disadvantage of complicating the expressions in the boundary conditions $(I)$ and $(II)$, the presence of the parameters also may bring some advantages, e.g.~allowing to simplify the conditions by a particular choice of the parameters or by taking limits whenever the terms in \eqref{def:L} admit suitable asymptotics.

We now  show that $D(\cL)$ is independent of the choice of $\mu\in\C^-$ and $\lambda\in\C^+$. First, we show this for the conditions to lie in the domains of the operators, and in a second lemma we consider $(I)$ and $(II)$.

\begin{lem}
Let $u\in H$, $v_-(0)\in E_*$ and $v_+(0)\in E$.
If $u+(\Gamma_*(A^*+\mu_0)^{-1})^*v_{-}(0)\in D(A)$ for some $\mu_0\in\C^-$, then $u+(\Gamma_*(A^*+\mu)^{-1})^*v_{-}(0)\in D(A)$ for all $\mu_0\in\C^-$.
Similarly, if $u+(\Gamma(A+\lambda_0)^{-1})^*v_{+}(0)\in D(A^*)$ for some $\lambda_0\in\C^+$, then $u+(\Gamma(A+\lambda)^{-1})^*v_{+}(0)\in D(A^*)$ for all $\lambda\in\C^+$.
\end{lem} 

\begin{proof}
We show that $u+(\Gamma_*(A^*+\mu)^{-1})^*v_{-}(0)\in D(A)$ for all $\mu_0\in\C^-$. Clearly, it suffices to show that 
$$\left((\Gamma_*(A^*+\mu)^{-1})^*-(\Gamma_*(A^*+\mu_0)^{-1})^*\right)v_{-}(0)= \left[\Gamma_*((A^*+\mu)^{-1}-(A^*+\mu_0)^{-1})\right]^*v_{-}(0) \in D(A).$$
We have 
\bea
\left[\Gamma_*((A^*+\mu)^{-1}-(A^*+\mu_0)^{-1})\right]^*v_{-}(0)&=& \left[(\mu_0-\mu)\Gamma_*((A^*+\mu)^{-1}(A^*+\mu_0)^{-1})\right]^*v_{-}(0)\\
&=& (\bar{\mu}_0-\bar{\mu})(A+\bar{\mu}_0)^{-1}\left(\Gamma_*(A^*+\mu)^{-1}\right)^*v_-(0),
\eea
which clearly lies in $D(A)$.
\end{proof}

Now we are in a position to prove independence of the boundary conditions  $(I)$ and $(II)$  from the choice of $\mu\in\C^-$ and $\lambda\in\C^+$.
\begin{lem}\label{lemma:2.2}
The conditions (I) and (II) in \eqref{def:L}
are independent of $\mu\in\C^{-}$ and $\lambda\in\C^+$, respectively.
\end{lem}

\begin{proof}
We check this for (I). We need to show that for any $\mu,\tilde{\mu}\in\Cc^{-}$ we have
\be\label{eq:zero} S^*(-\mu)-S^*(-\mut)+i\Gamma\left(\Gamma_*((A^*+\mu)^{-1}-(A^*+\mut)^{-1})\right)^*=0 \hbox{ on } E_*.\ee
Since, 
\bea
\Gamma\left(\Gamma_*((A^*+\mu)^{-1}-(A^*+\mut)^{-1})\right)^*&=&\Gamma\left((\mut-\mu)\Gamma_*((A^*+\mu)^{-1}(A^*+\mut)^{-1})\right)^*\\
&=&(\mutb-\mub) \left(\Gamma (A+\mutb)^{-1}\right)\left(\Gamma_*(A^*+\mu)^{-1}\right)^*,
\eea
the result follows from taking adjoints in \eqref{eq:Sdiff}. 
\end{proof}

The next lemma shows that the conditions $(I)$ and $(II)$ in \eqref{def:L} are equivalent.
\begin{lemma}\label{lem:5.4}
Let $\mu\in\C^-$, $\lambda\in\C^+$, $u\in H$, $v_+(0)\in E$ and $v_-(0)\in E_*$. 
Then $u+(\Gamma_*(A^*+\mu)^{-1})^*v_{-}(0)\in D(A)$ and 
$$ \mbox{(I)} \; v_+(0)=S^*(-\mu)v_-(0)+i\Gamma\left(u+(\Gamma_*(A^*+\mu)^{-1})^*v_{-}(0)\right)$$ holds if and only if
 $u+(\Gamma(A+\lambda)^{-1})^*v_{+}(0)\in D(A^*)$ and 
$$  \mbox{(II)} \; v_-(0)=S(-\bar{\lambda})v_+(0)-i\Gamma_*\left( u+(\Gamma(A+\lambda)^{-1})^*v_{+}(0)\right).$$
\end{lemma} 

\begin{proof}
We will assume that $u+(\Gamma(A+\lambda)^{-1})^*v_{+}(0)\in D(A^*)$ and $(II)$ holds. The proof of the converse is similar. First, we need to show that $w:=u+(\Gamma_*(A^*+\mu)^{-1})^*v_{-}(0)\in D(A)$. Now,
$$w=u+(\Gamma_*(A^*+\mu)^{-1})^*\left[S(-\bar{\lambda})v_+(0)-i\Gamma_*\left( u+(\Gamma(A+\lambda)^{-1})^*v_{+}(0)\right)\right].$$
Using \eqref{eq:Sgammastar}, we get
$$ w= u +\left[I+(\bar{\lambda}-\bar{\mu})(A+\bar{\mu})^{-1}\right]\left(\Gamma(A+\lambda)^{-1}\right)^* v_+(0) - i (\Gamma_*(A^*+\mu)^{-1})^*\Gamma_*\left( u+(\Gamma(A+\lambda)^{-1})^*v_{+}(0)\right).$$
Clearly, $w\in D(A)$ if and only if $\tilde{w}\in D(A)$, where
\bea
\tilde{w}&:=& u +\left(\Gamma(A+\lambda)^{-1}\right)^* v_+(0) - i (\Gamma_*(A^*+\mu)^{-1})^*\Gamma_*\left( u+(\Gamma(A+\lambda)^{-1})^*v_{+}(0)\right)\\
&=& \left[I-i (\Gamma_*(A^*+\mu)^{-1})^*\Gamma_*\right]\left(u +\left(\Gamma(A+\lambda)^{-1}\right)^* v_+(0)\right)
.
\eea
Next, using that $u+(\Gamma(A+\lambda)^{-1})^*v_{+}(0)\in D(A^*)$, for $\nu\in\C^+$ there exists $f$ such that 
$$u+(\Gamma(A+\lambda)^{-1})^*v_{+}(0)=(A^*-\nu)^{-1}f.$$
Inserting this in $\tilde{w}$ and using \eqref{GreenStar}, we get
\bea
\tilde{w}&=& (A^*-\nu)^{-1}f + \left[(A+\bar{\mu})^{-1}-(A^*-\nu)^{-1}+(\bar{\mu} +\nu)(A+\bar{\mu})^{-1} (A^*-\nu)^{-1}\right] f\\
&=& \left[(A+\bar{\mu})^{-1}+(\bar{\mu} +\nu)(A+\bar{\mu})^{-1} (A^*-\nu)^{-1}\right] f \in D(A).
\eea

It remains to check $(I)$, which using Lemma \ref{lemma:adj}, is equivalent to
\ben \label{eq:I}
v_+(0)-S^*(-\mu)S(-\bar{\lambda})v_+(0) &=& -iS_*(-\bar{\mu})\Gamma_*\left( u+(\Gamma(A+\lambda)^{-1})^*v_{+}(0)\right)\\
&& \hspace{-30pt}+i\Gamma\left(u+(\Gamma_*(A^*+\mu)^{-1})^*\left(S(-\bar{\lambda})v_+(0)-i\Gamma_*\left( u+(\Gamma(A+\lambda)^{-1})^*v_{+}(0)\right)\right)\right).
\nonumber
\een
By Lemma \ref{lem:SStarS}, the left hand side of \eqref{eq:I} is given by
\be\label{eq:lhs}
i(\bar{\lambda}-\bar{\mu})\left(\Gamma(A+\bar{\mu})^{-1}\right)\left(\Gamma(A+\lambda)^{-1}\right)^* v_+(0).
\ee
We now calculate the terms on the right hand side of \eqref{eq:I}. By the definition of $S_*$, we have
\be\label{eq:term1}
S_*(-\bar{\mu})\Gamma_*\left( u+(\Gamma(A+\lambda)^{-1})^*v_{+}(0)\right) =  \Gamma (A+\bar{\mu})^{-1}(A^*+
\bar{\mu})\left( u+(\Gamma(A+\lambda)^{-1})^*v_{+}(0)\right),
\ee
by \eqref{eq:Sgammastar} we have 
\ben\label{eq:term2}
(\Gamma_*(A^*+\mu)^{-1})^*S(-\bar{\lambda})v_+(0)&=& \left[I+(\bar{\lambda}-\bar{\mu})(A+\bar{\mu})^{-1}\right]\left(\Gamma(A+\lambda)^{-1}\right)^* v_+(0)\\
&=& \nonumber\left(\Gamma(A+\lambda)^{-1}\right)^* v_+(0) + (\bar{\lambda}-\bar{\mu})(A+\bar{\mu})^{-1}\left(\Gamma(A+\lambda)^{-1}\right)^* v_+(0) .
\een
Applying $i\Gamma$ to the second term (which lies in $D(A)$) gives the same as in \eqref{eq:lhs}, so this will precisely cancel the left hand side in \eqref{eq:I}. 
The remaining terms on the right hand side of \eqref{eq:I} now equal
\bea
&&-i\Gamma (A+\bar{\mu})^{-1}(A^*+\bar{\mu})(u+(\Gamma(A+\lambda)^{-1})^*v_{+}(0)) \\
&&+i\Gamma\left(u+\left(\Gamma(A+\lambda)^{-1}\right)^* v_+(0)-i(\Gamma_*(A^*+\mu)^{-1})^*\Gamma_*(u+(\Gamma(A+\lambda)^{-1})^*v_{+}(0))\right).
\eea
Setting $u+(\Gamma(A+\lambda)^{-1})^*v_{+}(0)=(A^*+\bar{\lambda})^{-1}f$, this is 
\bea
-i\Gamma \left[ (A+\bar{\mu})^{-1}(A^*+\bar{\mu})(A^*+\bar{\lambda})^{-1}f 
-(A^*+\bar{\lambda})^{-1}f+i(\Gamma_*(A^*+\mu)^{-1})^*\left(\Gamma_*(A^*+\bar{\lambda})^{-1}\right)f\right].
\eea
Applying \eqref{GreenStar} shows that this equals zero, as required.
\end{proof}

\begin{remark}
 This shows which are the free parameters in the description of $D(\cL)$: 
\begin{itemize}
	\item The value for $v_+(0)\in E$,
	\item a vector $h\in D(A^*)$ (such that $u=h-(\Gamma(A+\lambda)^{-1})^*v_+(0)$) and
	\item two vector-valued functions $w_\pm\in H^1(\R^{\pm})$ with $w_+(0)=0=w_-(0)$.
\end{itemize}
\end{remark}

So far, we have discussed the domain of the dilation. We now present two equivalent formulae for the part of the action of the dilation on the original Hilbert space $H$.
\begin{definition}
Let $\mu\in\C^-$ and $\lambda\in\C^+$. For $U\in D(\cL)$ as given in \eqref{def:L}, define two operators $T,T_*:D(\cL)\to H$ by
\be\label{eq:T}
TU:=A^*(u+(\Gamma(A+\lambda)^{-1})^*v_{+}(0))+\bar{\lambda}(\Gamma(A+\lambda)^{-1})^*v_{+}(0)
\ee
and
\be\label{eq:Tstar}
T_*U:=A(u+(\Gamma_*(A^*+\mu)^{-1})^*v_{-}(0))+\bar{\mu}(\Gamma_*(A^*+\mu)^{-1})^*v_{-}(0).
\ee
\end{definition}
Note that we have
\be\label{eq:T2}
TU=\left(A^*+\bar{\lambda}\right)(u+(\Gamma(A+\lambda)^{-1})^*v_{+}(0))-\bar{\lambda}u
\ee
and 
\be\label{eq:Tstar2}
T_*U=\left(A+\bar{\mu}\right)(u+(\Gamma_*(A^*+\mu)^{-1})^*v_{-}(0))-\bar{\mu}u.
\ee

\begin{lemma}\label{lem:T}
The operators $T$ and $T_*$ coincide on  $D(\cL)$. 
\end{lemma} 

\begin{proof}
Let $U\in D(\cL)$, $\lambda\in\C^+$ and $\mu\in \C^-$. We test the equality with functions $g$ from the dense set $D(A)$:
\bea
\llangle g, TU\rrangle &=& \llangle Ag, u\rrangle +\llangle \Gamma(A+\lambda)^{-1}) A g, v_{+}(0) \rrangle +\llangle \la \Gamma(A+\lambda)^{-1}g, v_+(0)\rrangle\\
 &=& \llangle Ag, u\rrangle +\llangle \Gamma g, v_{+}(0) \rrangle \\
 &=& \llangle Ag, u+ (\Gamma_*(A^*+\mu)^{-1})^*v_{-}(0) \rrangle - \llangle Ag,  (\Gamma_*(A^*+\mu)^{-1})^*v_{-}(0) \rrangle+\llangle \Gamma g, v_{+}(0) \rrangle.
\eea
Using the Lagrange identity \eqref{eq:Lagrange} for the first term gives
\bea
\llangle g, TU\rrangle &=& \llangle g,A( u+ (\Gamma_*(A^*+\mu)^{-1})^*v_{-}(0)) \rrangle +i\llangle \Gamma g,\Gamma( u+ (\Gamma_*(A^*+\mu)^{-1})^*v_{-}(0)) \rrangle \\
&& - \llangle Ag,  (\Gamma_*(A^*+\mu)^{-1})^*v_{-}(0) \rrangle+\llangle \Gamma g, v_{+}(0) \rrangle \\
&=&\llangle g, T_*U -\bar{\mu} (\Gamma_*(A^*+\mu)^{-1})^*v_{-}(0) \rrangle +i\llangle \Gamma g,\Gamma( u+ (\Gamma_*(A^*+\mu)^{-1})^*v_{-}(0)) \rrangle \\
&& - \llangle Ag,  (\Gamma_*(A^*+\mu)^{-1})^*v_{-}(0) \rrangle+\llangle \Gamma g, v_{+}(0) \rrangle\\
&=&\llangle g, T_*U \rrangle  -\mu \llangle  g, (\Gamma_*(A^*+\mu)^{-1})^* v_{-}(0) \rrangle +\llangle \Gamma g, v_+(0) - i\Gamma( u+ (\Gamma_*(A^*+\mu)^{-1})^*v_{-}(0)) \rrangle \\
&& - \llangle Ag,  (\Gamma_*(A^*+\mu)^{-1})^*v_{-}(0) \rrangle \\
&=&\llangle g, T_*U \rrangle  - \llangle (\Gamma_*(A^*+\mu)^{-1}) (A+\mu) g,  v_{-}(0) \rrangle +\llangle \Gamma g, S^*(-\mu) v_-(0) \rrangle,
\eea 
where the last equality follows from the boundary condition $(I)$ in $D(\cL)$. Therefore, we get 
\bea
\llangle g, TU\rrangle- \llangle g, T_*U \rrangle&=&   \llangle S(-\mu)\Gamma g- (\Gamma_*(A^*+\mu)^{-1}) (A+\mu) g,  v_{-}(0) \rrangle =0,
\eea
by definition of the characteristic function. Hence $TU=T_*U$.
\end{proof}

The previous result immediately shows the following corollary, which justifies the absence of the parameters $\la$ and $\mu$ in our notation of $T$ and $T_*$.
\begin{corollary}\label{cor:T}
$T$ and $T_*$ on $D(\cL)$ are independent of $\la$ and $\mu$.
\end{corollary}

Finally, having defined the domain on the dilation $D(\cL)$ in Definition \ref{def:L}, we can now give its full action. 

\begin{definition}
We define the operator $\cL$ on $\cH$ with domain $D(\cL)$ by
$$\cL U= \cL\U=\left(\begin{array}{c} iv'_{-} \\ TU \\ iv'_{+} \end{array}\right).$$
\end{definition} 

We see that in the so-called incoming and outgoing channels (the first and last components), see, e.g.~\cite{LP67}, the operator $\cL$ is a simple first order differentiation operator, while on the part in $H$, it is given by $T$ or $T_*$, which act essentially like $A^*$ or $A$ with correction or coupling terms from the channels.

\section{Examples} 
In this section we consider some special cases for which we  determine the operator $\cL$ and its domain more explicitly.

\subsection{The case of bounded imaginary part}
We start with a very simple well-known example. 

\begin{lemma}
Assume $A=\Re A+ i \Im A$ with $\Re A$ a selfadjoint operator, $\Im A$ a bounded non-negative operator and $D(A)=D(A^*)=D(\Re A)$. We can choose $\Gamma=\Gamma_*=\sqrt{2\Im A}$ and $E=E_*=\overline{\Ran\Im A}\subseteq H$.
Then $U=\U\in D(\cL)$ if and only if $u\in D(A)$ and $v_+(0)=v_-(0)+i\Gamma u$.
\end{lemma}

\begin{proof}
It is easy to check \eqref{eq:Lagrange} and \eqref{eq:LagrangeStar} hold with the given $\Gamma,\Gamma_*$, $E$ and $E_*$.
 In particular, $\Gamma $ and $\Gamma_*$ are bounded, so for any $v_+(0), v_-(0)$ we have
$$ (\Gamma(A+\lambda)^{-1})^*v_{+}(0) = (A^*+\bar{\lambda})^{-1}\Gamma^* v_+(0)\in D(A^*)\quad \hbox{ and }\quad  (\Gamma_*(A^*+\mu)^{-1})^*v_{-}(0) = (A+\bar{\mu})^{-1}\Gamma_*^* v_-(0)\in D(A) .$$
This shows that whenever $U\in D(\cL)$ we have $u\in D(A)$ and so 
\be\label{eq:vplus}
v_+(0)=S^*(-\mu)v_-(0)+i\Gamma (\Gamma_*(A^*+\mu)^{-1})^*v_{-}(0) + i\Gamma u.
\ee
Next let $v_-(0)$ in $\Ran\Gamma_*$ be arbitrary, $f\in D(A^*)$ with $v_-(0)=\Gamma_*f$  and $g=(A^*+\mu)f$. Then using the definition of $S_*$ and  \eqref{GreenStar}, we have
\ben\label{eq:id}
&& S^*(-\mu)v_-(0)+i\Gamma (\Gamma_*(A^*+\mu)^{-1})^*v_{-}(0)\\ \nonumber &=& \left[S_*(-\bar{\mu})+i\Gamma (\Gamma_*(A^*+\mu)^{-1})^*\right]\Gamma_*(A^*+\mu)^{-1}g\\ \nonumber
&=& \Gamma\left[ (A+\bar{\mu})^{-1}(A^*+\bar{\mu})(A^*+\mu)^{-1} - (A+\bar{\mu})^{-1} + (A^*+\mu)^{-1} - (\bar{\mu}-\mu) (A+\bar{\mu})^{-1}\right]g\\ \nonumber
&=& \Gamma f \ = \ v_-(0).
\een
Since this holds for $v_-(0)$ from the dense set $\Ran\Gamma_*$,  it holds on the whole space $E_*$ and \eqref{eq:vplus} reduces to $v_+(0)=v_-(0)+i\Gamma u$.

On the other hand, if $u\in D(A)$  then clearly  the domain inclusions needed in \eqref{def:L} are satisfied and $(I)$ follows from $v_+(0)=v_-(0)+i\Gamma u$ by using \eqref{eq:id}.
\end{proof}

We note that similar considerations work for the case of relatively bounded imaginary part.

\subsection{Dissipative Schr\"{o}dinger operators on the half line}

This section considers the combination of dissipative boundary conditions and potentials for Schr\"{o}dinger operators, providing an example where the imaginary part of the operator is not bounded. It also illustrates the usefulness of being able to consider limits of the parameters $\la$ and $\mu$ in the description of the dilation.

{We now consider the  Schr\"{o}dinger operator $A$ in $L^2(\R^+)$ as discussed in Example \ref{ex:bcq}. We note that the operator $A$ is completely non-selfadjoint provided either $\Im(h)\neq 0$ or $\Im(q)$ is not identically zero. We will prove this in forthcoming work.} Our first aim is to determine the asymptotics of the characteristic function given in \eqref{SchrCharFn}. To this end, we begin with an elementary abstract lemma.

\begin{lem}\label{lem:id}
Let $S_n=\twomat{\alpha_n}{\beta_n}{\gamma_n}{\delta_n}$, $n\in\N$,  be a sequence of 2$\times$2 block operator matrices in $H_1\oplus H_2$ consisting of  bounded operators and such that $S_n$ is a contraction for all $n\in\N$. If $\alpha_n\stackrel{s}{\rightarrow} I_{H_1}$ and  $\delta_n\stackrel{s}{\rightarrow} I_{H_2}$, then  $\beta_n,\gamma_n\stackrel{s}{\rightarrow} 0$. Here  $A_n\stackrel{s}{\rightarrow} A$ denotes strong convergence. 
\end{lem}

\begin{proof}
We prove the statement for $\gamma_n$. The proof for $\beta_n$ is similar.
 For any $x\in H_1$,
$$0\leq\norm{\alpha_n^*\alpha_nx-x}^2=\norm{\alpha_n^*\alpha_nx}^2+\norm{x}^2-2\norm{\alpha_nx}^2 \leq 2\norm{x}^2-2\norm{\alpha_nx}^2\to 0,$$ by assumption on $\alpha_n$. Thus, $\alpha_n^*\alpha_n\stackrel{s}{\rightarrow} I_{H_1}$.
However,
$S_n^*S_n$ is also a contraction, so $0\leq \alpha_n^*\alpha_n+\gamma_n^*\gamma_n\leq I_{H_1}$. This implies $\gamma_n^*\gamma_n\stackrel{s}{\rightarrow} 0$. Then for any $x\in H_1$, 
$$\norm{\gamma_n x}^2=\llangle \gamma_n^*\gamma_n x, x \rrangle \to 0,$$
proving $\gamma_n\stackrel{s}{\rightarrow} 0$.
\end{proof}

Using that $m_*(z)\sim i\sqrt{z}$ as $z\to +i\infty$ (see \cite{Eve72}) and using the resolvent estimate for the anti-dissipative operator $A^*$, we see that the two diagonal terms in the characteristic function in \eqref{SchrCharFn} converge strongly to the identity. By Lemma \ref{lem:id}, the two off diagonal terms must converge strongly to $0$. Therefore, $S(z) \stackrel{s}{\rightarrow} I$ as  $z\to +i\infty$.

To determine the conditions for lying in $D(\cL)$ more explicitly, we next determine $(\Gamma_*(A^*-z)^{-1})^*$. 
 Let $G_*^0$ denote the Green function associated with $A^*$ given by
\be\label{SchrGreen} G_*^0(x,y;z)=\frac{1}{\bar{h}-m_*(z)}
\begin{cases} (\bar{h} \varphi_*(x,z)+\psi_*(x,z)) (m_*(z) \varphi_*(y,z)+\psi_*(y,z)), & x<y, \\ (\bar{h} \varphi_*(y,z)+\psi_*(y,z)) (m_*(z) \varphi_*(x,z)+\psi_*(x,z)), & x>y, \end{cases}  \ee
with 
$\varphi_*$, $\psi_*$ the fundamental solutions from Example \ref{ex:bcq}.
 Thus
\be\label{term1}
 (\Gamma_*(A^*-z)^{-1})^*v_{-}(0) =\sqrt{2\Im h}\  \overline{G_*^0(0,\cdot,z)} v_-(0).
\ee

Let $c\in\C$, $\eta\in L^2(\{\Im q>0\})$ and  $p\in L^2(\R^+)$. Then
\bea
\llangle \left(\Gamma_*(A^*-z)^{-1}\right)^* \twovec{c}{\eta(x)}, p(x)\rrangle &=& \llangle  \twovec{c}{\eta(x)}, \Gamma_*(A^*-z)^{-1} p(x)\rrangle \\
&=& \llangle  \twovec{c}{\eta(x)},\twovec{\sqrt{2\Im h} \int_0^\infty G_*(0,y,z)p(y)\ dy}{\sqrt{2\Im q(x)} \int_0^\infty G_*(x,y,z)p(y)\ dy} \rrangle \\
&=& \llangle c \sqrt{2\Im h} \overline{G_*(0,y,z)} + (A-\bar{z})^{-1}\sqrt{2\Im q(x)}\eta(x), p(y)\rrangle,
\eea
so
$$\left(\Gamma_*(A^*-z)^{-1}\right)^* \twovec{c}{\eta}=c \sqrt{2\Im h} \overline{G_*(0,\cdot,z)} + (A-\bar{z})^{-1}\sqrt{2\Im q}\eta.$$
Therefore, the condition
$u+(\Gamma_*(A^*+\mu)^{-1})^*v_{-}(0)\in D(A)$ becomes 
$$u+ (v_-(0))_1 \sqrt{2\Im h} \overline{G_*(0,y,z)}= u+ (v_-(0))_1 \sqrt{2\Im h}\frac{\overline{m_*(z)\varphi(y)+\psi(y)}}{h-\overline{m_*(z)}} \in D(A),$$  where $(v_-(0))_1$ denotes the first component of $v_-(0)$. Evaluating the expression and its derivative at $0$ gives the condition $$u'(0)-hu(0)=(v_-(0))_1 \sqrt{2\Im h}.$$

Next, we consider the boundary condition $v_+(0)=S^*(-\mu)v_-(0)+i\Gamma\left(u+(\Gamma_*(A^*+\mu)^{-1})^*v_{-}(0)\right)$, as $\mu\to-i\infty$. By  $\hat{\Gamma}$ we denote the extension of $\Gamma_*$ from $D(A^*)$ to $H^1(\R^+)$. By the calculation above we have that 
$$\hat{\Gamma}(\Gamma_*(A^*+\mu)^{-1})^*v_{-}(0) = \hat{\Gamma}\left[ (v_-(0))_1 \sqrt{2\Im h} \overline{G_*(0,\cdot,-\mu)} + (A+\bar{\mu})^{-1}\sqrt{2\Im q}(v_-(0))_2\right].$$
Now, noting that $\overline{G_*(x,y,z)}=G(x,y,\bar{z})$, $G_*(0,0,z)=(h-m_*(z))^{-1}\to 0$, $\norm{G(0,y,-\bar{\mu})}_2\to 0 $ and 
$$\norm{G_*(0,y,z)}_2\leq \underbrace{\norm{G_*^0(0,y,z)}_2}_{\to 0}\left(1+\underbrace{\norm{qG_*(0,y,z)}_2}_{\sim(\Im z)^{-1}}\right),$$
we see that this term will vanish in the limit. Therefore, we obtain
\bea
 D(\cL) & = & \left\{ U=\U \; : \; u\in H^2(\R^+),\; v_\pm\in H^1(\R_{\pm}), \right. \\
 & & \;\;\; u'(0)-hu(0)=\sqrt{2\Im h}\ (v_-(0))_1, \; v_+(0)=v_-(0)+i\hat{\Gamma}u \Bigg\} . \nonumber 
\eea

Finally, we consider the action of $\cL$. We have
\bea
T_*\U&=&(A+\mub)(u+(\Gamma_*(A^*+\mu)^{-1})^*v_{-}(0))-\mub u \\
&=& (A+\mub)\left(u+ \frac{\sqrt{2\Im h} (v_-(0))_1}{h-\overline{m(-\mu)}} \overline{(m(-\mu) \varphi(\cdot,-\mu)+\psi(\cdot,-\mu))}  + (A+\mub)^{-1}\sqrt{2\Im q}\ (v_-(0))_2 \right) - \mub u \\
&=& -u''+qu+\sqrt{2\Im q}\ (v_-(0))_2,
\eea
which shows in particular by explicit calculation that $T_*$ here is independent of 
$\mu$. 
We now have
 $$\cL U= \cL\U=\left(\begin{array}{c} iv'_{-} \\ -u''+qu+\sqrt{2\Im q}\ (v_-(0))_2 \\ iv'_{+} \end{array}\right).$$

\begin{remark}
The dilation property of $\cL$ can easily be checked: Let $\lambda\in\C^+$ and
$$\U=(\cL-\lambda)^{-1}{\left(\begin{array}{c} 0 \\ w \\ 0 \end{array}\right)}.$$ Then $v_+(x)=v_+(0)e^{-i\lambda x}$, so $v_+(0)$ must vanish. Thus $v_-(0)=-i\hat{\Gamma}u$ and so 
\bea
T_*\U&=& -u''+qu+\sqrt{2\Im q}\ (v_-(0))_2 \ =\ -u''+qu +\sqrt{2\Im q} (-i\sqrt{2\Im q}u) \ =\ -u''+\bar{q}u,
\eea
giving $w=-u''+\bar{q}u$.
Therefore, $P_H \left(\cL-\la\right)^{-1}\vert_H= (A^*-\la)^{-1}$
Similarly, one can see that for $\lambda$ in the lower half plane one has $P_H \left(\cL-\la\right)^{-1}\vert_H= (A-\la)^{-1}$.
We will see later in Theorem \ref{Th:dil} that this is a general property of the operator $\cL$ we have constructed.
\end{remark}

\section{Properties of $\cL$}

We first calculate the resolvent of the operator $\cL$.
\begin{lemma}\label{lem:res}
For $\W\in\cH$ and $\lambda_0\in\C^+$, we have
\be\label{eq:res+}
(\cL-\lambda_0)^{-1}\W = \left(\begin{array}{c} -i\Gamma_*(A^*-\lambda_0)^{-1} w e^{-i\lambda_0x} + i S(\lambda_0)\int_0^\infty e^{i\lambda_0(t-x)}g(t)\ dt  -i\int_0^x e^{i\lambda_0(t-x)}f(t)\ dt \\ 
(A^*-\lambda_0)^{-1}w-i\left(\Gamma(A-\overline{\lambda_0})^{-1}\right)^*\int_0^\infty e^{i\lambda_0 t}g(t)\ dt \\ i\int_x^\infty e^{i\lambda_0(t-x)}g(t)\ dt \end{array}\right).
\ee
Similarly, for $\lambda_0\in\C^-$, we have
\be\label{eq:res-}
(\cL-\lambda_0)^{-1}\W = \left(\begin{array}{c} -i\int_{-\infty}^x e^{i\lambda_0(t-x)}f(t)\ dt \\ 
(A-\lambda_0)^{-1} w +i (\Gamma_*(A^*-\overline{\lambda_0})^{-1})^*\int_{-\infty}^0 e^{i\lambda_0t}f(t)\ dt \\ 
i e^{i\lambda_0 x} \Gamma (A-\lambda_0)^{-1}w  - i  S^*(\overline{\lambda_0}) \int_{-\infty}^0 e^{i\lambda_0(t-x)}f(t)\ dt
-i\int_0^x e^{i\lambda_0(t-x)}g(t)\ dt \end{array}\right).
\ee

\end{lemma}

\begin{proof}
We prove \eqref{eq:res-}, the proof of \eqref{eq:res+} is similar. Let $\U\in D(\cL)$, $\lambda_0\in\C^-$ and 
$$(\cL-\lambda_0)\U= \left(\begin{array}{c} iv_-'-\lambda_0v_- \\ 
A(u+(\Gamma_*(A^*+\mu)^{-1})^*v_{-}(0))+\bar{\mu}(\Gamma_*(A^*+\mu)^{-1})^*v_{-}(0)-\lambda_0 u\\
iv_+'-\lambda_0v_+
 \end{array}\right)  = \W,$$
where $\mu\in\C^-$ is arbitrary. Choosing $\mu=-\overline{\lambda_0}$, this simplifies to
$$ \left(\begin{array}{c} iv_-'-\lambda_0v_- \\ 
(A-\lambda_0)(u+(\Gamma_*(A^*-\overline{\lambda_0})^{-1})^*v_{-}(0))\\
iv_+'-\lambda_0v_+
 \end{array}\right)  = \W.$$
We can easily solve the first and last equation, taking into account that $\Im\lambda_0<0$, we get that
\be\label{channelsolns}v_-(x)=-i\int_{-\infty}^x e^{i\lambda_0(t-x)}f(t)\ dt\quad\hbox{ and }\quad v_+(x)=v_+(0)e^{i\lambda_0 x}-i\int_0^x e^{i\lambda_0(t-x)}g(t)\ dt.\ee
Solving the second equation for $u$ gives that
$$u=(A-\lambda_0)^{-1} w - (\Gamma_*(A^*-\overline{\lambda_0})^{-1})^*v_{-}(0)=(A-\lambda_0)^{-1} w +i (\Gamma_*(A^*-\overline{\lambda_0})^{-1})^*\int_{-\infty}^0 e^{i\lambda_0t}f(t)\ dt .$$
It remains to determine $v_+(0)$ from
\bea
v_+(0)&=& S^*(\overline{\lambda_0})v_-(0)+i\Gamma\left(u+(\Gamma_*(A^*-\overline{\lambda_0})^{-1})^*v_{-}(0)\right)\\
&=& i\Gamma (A-\lambda_0)^{-1}w - i S^*(\overline{\lambda_0}) \int_{-\infty}^0 e^{i\lambda_0t}f(t)\ dt.
\eea
Inserting this in \eqref{channelsolns} proves the result.
\end{proof}

\begin{lemma}\label{lem:symm}
$\cL$ is symmetric.
\end{lemma}

\begin{proof}
Let $U=\U, \widetilde{U}=\Ut\in D(\cL)$. Then
\be\label{eq:symm}
\llangle \cL U, \widetilde{U} \rrangle - \llangle  U, \cL\widetilde{U} \rrangle = i \left( \llangle v_-(0), \widetilde{v_-}(0) \rrangle -\llangle v_+(0), \widetilde{v_+}(0) \rrangle  \right) +\llangle T_*U,\widetilde{u}\rrangle - \llangle u, T\widetilde{U}\rrangle.
\ee
Then, for $\lambda\in\C^+$ we have
\bea
\llangle T_*U,\widetilde{u}\rrangle - \llangle u, T\widetilde{U}\rrangle &=& 
\llangle (A+\lambda) (u+(\Gamma_*(A^*+\overline{\lambda})^{-1})^*v_{-}(0))-\lambda u, \widetilde{u}\rrangle - 
\llangle u, (A^*+\bar{\lambda}) (\widetilde{u}+(\Gamma(A+\lambda)^{-1})^*\widetilde{v_{+}}(0))
-\bar{\lambda}\widetilde{u}\rrangle\\
&=& 
\llangle (A+\lambda) (u+(\Gamma_*(A^*+\overline{\lambda})^{-1})^*v_{-}(0)), \widetilde{u}\rrangle - 
\llangle u, (A^*+\bar{\lambda}) (\widetilde{u}+(\Gamma(A+\lambda)^{-1})^*\widetilde{v_{+}}(0))\rrangle\\
&=& 
\llangle (A+\lambda) (u+(\Gamma_*(A^*+\overline{\lambda})^{-1})^*v_{-}(0)), \widetilde{u}\rrangle \\
&&
+ \llangle (\Gamma_*(A^*+\overline{\lambda})^{-1})^*v_{-}(0), (A^*+\bar{\lambda}) (\widetilde{u}+(\Gamma(A+\lambda)^{-1})^*\widetilde{v_{+}}(0))\rrangle\\
&& - 
\llangle u+(\Gamma_*(A^*+\overline{\lambda})^{-1})^*v_{-}(0), (A^*+\bar{\lambda}) (\widetilde{u}+(\Gamma(A+\lambda)^{-1})^*\widetilde{v_{+}}(0))\rrangle\\
&=& 
 \llangle (\Gamma_*(A^*+\overline{\lambda})^{-1})^*v_{-}(0), (A^*+\bar{\lambda}) (\widetilde{u}+(\Gamma(A+\lambda)^{-1})^*\widetilde{v_{+}}(0))\rrangle\\
&& - 
\llangle (A+{\lambda}) (u+(\Gamma_*(A^*+\overline{\lambda})^{-1})^*v_{-}(0)),  (\Gamma(A+\lambda)^{-1})^*\widetilde{v_{+}}(0)\rrangle\\
&=& 
 \llangle v_{-}(0), \Gamma_* (\widetilde{u}+(\Gamma(A+\lambda)^{-1})^*\widetilde{v_{+}}(0))\rrangle - 
\llangle \Gamma (u+(\Gamma_*(A^*+\overline{\lambda})^{-1})^*v_{-}(0)),  \widetilde{v_{+}}(0)\rrangle.
\eea 
Using the conditions $(I)$ and $(II)$ in \eqref{def:L}, we get 
\bea
\llangle T_*U,\widetilde{u}\rrangle - \llangle u, T\widetilde{U}\rrangle &=& 
 \llangle v_{-}(0), i\widetilde{v_-}(0)-iS(-\bar{\lambda})\widetilde{v_{+}}(0)\rrangle - 
\llangle -iv_+(0)+iS^*(-\bar{\lambda})v_-(0),  \widetilde{v_{+}}(0)\rrangle\\
 &=& 
 \llangle v_{-}(0), i\widetilde{v_-}(0)\rrangle - 
\llangle -iv_+(0),  \widetilde{v_{+}}(0)\rrangle.
\eea 
Together with \eqref{eq:symm}, this proves symmetry of $\cL$.
\end{proof}

Combining the two previous results immediately gives:

\begin{corollary}
$\cL$ is selfadjoint.
\end{corollary}
\begin{proof}
By Lemma \ref{lem:symm}, $\cL$ is a symmetric operator, while by Lemma \ref{lem:res} we have that $\Ran (\cL-\lambda)=\cH$ for all non-real $\lambda$, i.e.~$\cL$ is a symmetric operator with deficiency indices $(0,0)$. Hence it is selfadjoint.
\end{proof}

\begin{theorem}\label{Th:dil}
$\cL$ is a minimal selfadjoint dilation of $A$. In particular, letting $P_H:\Hc\to H$ be the projection onto the second component, we have
\be\label{eq:dil} P_H \left(\cL-\la\right)^{-1}\vert_H=\begin{cases} (A-\la)^{-1} & \la\in\C^-, \\ (A^*-\la)^{-1} & \la\in\C^+.\end{cases}\ee
\end{theorem}

\begin{proof}
The formula \eqref{eq:dil} follows from \eqref{eq:res+} and  \eqref{eq:res-} by setting $f=g=0$ and considering the second component.

It remains to show minimality of the dilation. We need to show that 
$$\cH=\bigvee_{\lambda\not\in\R} (\cL-\lambda)^{-1} \left(\begin{array}{c} 0 \\ H \\ 0 \end{array}\right),$$
where $\bigvee$ denotes the closed linear span.
Let $$\U\perp\bigvee_{\lambda\not\in\R} (\cL-\lambda)^{-1} \left(\begin{array}{c} 0 \\ H \\ 0 \end{array}\right).$$
As $\cL$ is selfadjoint, $-\lambda (\cL-\lambda)^{-1}$ converges strongly to the identity and so $u=0$. Thus we get that 
$$\Hzero:=  (\cL-\lambda)^{-1} \Uzero \perp \left(\begin{array}{c} 0 \\ H \\ 0 \end{array}\right).$$
Then 
$$\Uzero = (\cL-\lambda) \Hzero,\quad \hbox{in particular}\quad T\Hzero=0=T_*\Hzero,$$
which implies from the definition of $T$ and $T_*$ that
\be\label{eq:star}0 = (A^*+\overline{\lambda_0}) (\Gamma(A+\lambda_0)^{-1})^*h_{+}(0), \quad\hbox{so}\quad (\Gamma(A+\lambda_0)^{-1})^*h_{+}(0)=0 \quad \hbox{for any}\quad \lambda_0\in\C^+
\ee
and
\be\label{eq:starstar}
0=(A+\mub)(\Gamma_*(A^*+\mu)^{-1})^*h_{-}(0),\quad\hbox{so}\quad (\Gamma_*(A^*+\mu)^{-1})^*h_{-}(0)=0 \quad \hbox{for any}\quad \mu\in\C^-.
\ee
From the density of the ranges of $\Gamma:D(A)\to E$ and $\Gamma_*:D(A^*)\to E_*$, this implies $h_+(0)=h_-(0)=0$.

On the other hand, $ih_+'-\lambda h_+=v_+$ implies that 
$$h_+(t)=e^{-i\lambda t}h_+(0)-ie^{-i\lambda t}\int_0^t e^{i\lambda x}v_+(x)\ dx.$$
Since for $\lambda\in\C^+$, the function $e^{-i\lambda t}$ is growing, we get  for these $\lambda$ that $h_+(0)=i\int_0^\infty e^{i\lambda x}v_+(x)\ dx=i\hat{v}_+(\lambda)$. 
Similarly for $\lambda\in\C^-$, we get that $h_-(0)=-i\hat{v}_-(\lambda)$.

Thus the Fourier transforms of $v_+$ and $v_-$ vanish in $\C^+$ and $\C^-$, respectively, showing that $v_+=0=v_-$, as desired.
\end{proof}

We complete this section with a discussion of complete non-selfadjointness. We start with  a lemma showing independence from parameters of developing certain ranges by the resolvent.

\begin{lemma}\label{lemma:cns}
For any $\lambda',\lambda''\in\C^+$ we have that
\ben\nonumber
&&\clos\left[\Span_{\mu\in\C^-}(A-\mu)^{-1} \left(\Ran\left(\left(\Gamma(A+i)^{-1}\right)^*\right)\cup \Ran\left(\left(\Gamma_*(A^*-i)^{-1}\right)^*\right)\right)\right.\\ 
\nonumber
&& \left. \quad +\
\Span_{\lambda\in\C^+}(A^*-\lambda)^{-1} \left(\Ran\left(\left(\Gamma(A+i)^{-1}\right)^*\right)\cup \Ran\left(\left(\Gamma_*(A^*-i)^{-1}\right)^*\right)\right)\right] \\ 
\label{eq:cns}
& = & \clos\left[\Span_{\mu\in\C^-}(A-\mu)^{-1} \left(\Ran\left(\left(\Gamma(A+\lambda')^{-1}\right)^*\right)\cup \Ran\left(\left(\Gamma_*(A^*-\lambda'')^{-1}\right)^*\right)\right)\right.\\ \nonumber
&&\left. \quad +\
\Span_{\lambda\in\C^+}(A^*-\lambda)^{-1} \left(\Ran\left(\left(\Gamma(A+\lambda')^{-1}\right)^*\right)\cup \Ran\left(\left(\Gamma_*(A^*-\lambda'')^{-1}\right)^*\right)\right)\right].
\een
\end{lemma}

\begin{proof}
We will show that the set on the right hand side of \eqref{eq:cns} is contained in the set on the left hand side by considering the difference between two typical terms. 
Consider 
\bea
(A-\mu)^{-1} \left[\left(\Gamma(A+i)^{-1}\right)^*- \left(\Gamma(A+\lambda')^{-1}\right)^*\right]&=& 
(A-\mu)^{-1} \left[\Gamma\left((A+i)^{-1}-(A+\lambda')^{-1}\right)\right]^* \\
& = & (\overline{\lambda'}+i) (A-\mu)^{-1} \left[\Gamma\left((A+i)^{-1}(A+\lambda')^{-1}\right)\right]^*\\
& = & (\overline{\lambda'}+i) (A-\mu)^{-1} (A^*+\overline{\lambda'})^{-1} \left[\Gamma(A+i)^{-1}\right]^*,
\eea
where we have used the Hilbert identity. Next, we use \eqref{GreenStar} to obtain 
\ben \nonumber &&
(A-\mu)^{-1} \left[\left(\Gamma(A+i)^{-1}\right)^*- \left(\Gamma(A+\lambda')^{-1}\right)^*\right]\\ \label{eq:cns2}&&= 
 \frac{\overline{\lambda'}+i}{\overline{\lambda'}+\mu} \left[(A-\mu)^{-1}-(A^*+\overline{\lambda'})^{-1}+i\left(\Gamma_*(A^*-\mub)^{-1}\right)^*\left(\Gamma_*(A^*+\overline{\lambda'})^{-1}\right)   \right]  \left[\Gamma(A+i)^{-1}\right]^*.
\een
Clearly, the first two terms on the right lie in the desired set. For the last term, we note the following two facts: Since $-\mu(A-\mu)^{-1}\to I$ as $\mu\to -i\infty$, the set on the left hand side of \eqref{eq:cns} contains  $\Ran\left(\left(\Gamma_*(A^*-i)^{-1}\right)^*\right)$ and by the Hilbert identity we have
$$(A-\mu)^{-1} \left(\Gamma_*(A^*-i)^{-1}\right)^* = \left(\Gamma_*(A^*-i)^{-1}(A^*-\mub)^{-1}\right)^* = \frac{1}{i-\mub} \left[ \left(\Gamma_*(A^*-i)^{-1}\right)^* - \left(\Gamma_*(A^*-\mub)^{-1}\right)^* \right].$$
Thus $\Ran \left(\Gamma_*(A^*-\mub)^{-1}\right)^*\subseteq \bigvee\left(\Ran\left( (A-\mu)^{-1} \left(\Gamma_*(A^*-i)^{-1}\right)^*\right) \cup \Ran\left(\left(\Gamma_*(A^*-i)^{-1}\right)^*\right)\right)$, showing that all terms on the right hand side of \eqref{eq:cns2} lie in the set on the left hand side of \eqref{eq:cns}. 

All other inclusions of terms on the right of  \eqref{eq:cns} in the set on the left can be checked similarly. The reverse inclusion follows in a similar manner.
\end{proof}

In the following we present a construction of the Langer decomposition from Proposition \ref{prop:Langer} and show its relation to the dilation.

\begin{Theorem}\label{thm:cns}
Let $A$ be a maximal dissipative operator and denote 
{
\ben
H_{cns}&=& \clos\left[\Span_{ \mu\in\C^-}\left\{(A-\mu)^{-1} \left(\Ran\left(\left(\Gamma(A+i)^{-1}\right)^*\right)\bigcup \Ran\left(\left(\Gamma_*(A^*-i)^{-1}\right)^*\right)\right)\right\}\right.\nonumber\\
&&\left. \quad +\  \Span_{\lambda\in\C^+}\left\{ (A^*-\lambda)^{-1} \left(\Ran\left(\left(\Gamma(A+i)^{-1}\right)^*\right)\bigcup \Ran\left(\left(\Gamma_*(A^*-i)^{-1}\right)^*\right)\right)\right\}\right].
\een
}
Let $H_{sa}=H\ominus H_{cns}$.
Then
\begin{enumerate}
	\item\label{i} $H_{cns}$ is a reducing subspace for $A$. 
	\item\label{ii}  $A_{sa}:=A\vert_{H_{sa}}$ is selfadjoint.
	\item\label{iii} The operator $A_{cns}:=A\vert_{H_{cns}}$ is completely non-selfadjoint.
	{	\item\label{iv} The subspace $ \left(\begin{array}{c} 0 \\ H_{sa} \\ 0 \end{array}\right)$ is reducing for the dilation $\cL$ with $\cL$ restricted to $ \left(\begin{array}{c} 0 \\ H_{sa} \\ 0 \end{array}\right)$ given by $A_{sa}$.
	\item\label{v} We have
	$$\bigvee_{\lambda\not\in\R} (\cL-\lambda)^{-1} \left(\begin{array}{c} L^2(\R^-,E_*) \\ 0 \\ L^2(\R^+,E) \end{array}\right)= \cH\ominus \left(\begin{array}{c} 0 \\ H_{sa} \\ 0 \end{array}\right)=\left(\begin{array}{c} L^2(\R^-,E_*) \\ H_{cns} \\ L^2(\R^+,E) \end{array}\right).$$}
	\item\label{vi}  $\cL$ restricted to $\left(\begin{array}{c} L^2(\R^-,E_*) \\ H_{cns} \\ L^2(\R^+,E) \end{array}\right)$ is a minimal selfadjoint dilation of $A_{cns}$.
\end{enumerate}
\end{Theorem}

\begin{remark}
For the case of bounded imaginary part, this result is known and can be found in \cite{Nab81}. 
\end{remark}

\begin{proof}
\begin{enumerate}
	\item We show that $(A-\mu_0)^{-1}H_{cns}\subseteq H_{cns}$ for all $\mu_0\in \C^-$. Similarly, one can show that $(A^*-\lambda_0)^{-1}H_{cns}\subseteq H_{cns}$ for all $\lambda_0\in \C^+$. Together, this shows that $H_{cns}$ is reducing for the resolvent of $A$, which implies it is reducing for $A$.
	
	We consider $w\in H_{cns}$ of the form
	\bea w&=& (A-\mu_1)^{-1} \left(\Gamma(A+i)^{-1}\right)^* h_1 + (A-\mu_2)^{-1}\left(\Gamma_*(A^*-i)^{-1}\right)^*h_2^*\\
	&& +
(A^*-\lambda_2)^{-1} \left(\Gamma(A+i)^{-1}\right)^*h_2 + (A^*-\lambda_1)^{-1} \left(\Gamma_*(A^*-i)^{-1}\right)^*h_1^*=:w_1+w_2+w_3+w_4,
\eea
	where $\mu_1,\mu_2\in\C^-$, $\lambda_1,\lambda_2\in \C^+$, $h_1,h_2\in E$ and $h_1^*,h_2^*\in E_*$. It is sufficient to show $(A-\mu_0)^{-1}w_i\in H_{cns}$ for $i=1,...,4$,  since linear combinations of vectors of this form are dense and the resolvent bounded.
	
	It is immediately clear from the Hilbert identity that $(A-\mu_0)^{-1}w_1, (A-\mu_0)^{-1}w_2\in H_{cns}$. For $(A-\mu_0)^{-1}w_3, (A-\mu_0)^{-1}w_4$, we use \eqref{GreenStar} and  the fact that $-\mu(A-\mu)^{-1}\to I$ as $\mu\to -i\infty$ to prove the inclusion.
	\item Let $u\in H_{sa}$. 
	By Lemma \ref{lemma:cns} and again using that $-\mu(A-\mu)^{-1}\to I$ as $\mu\to -i\infty$, this implies that for any $\mu\in\C^-$ we have
	$$ u \perp \Ran\left(\left(\Gamma_*(A^*+\mu)^{-1}\right)^*\right), \hbox{ or equivalently, } \Gamma_*(A^*+\mu)^{-1}u=0. $$
	By a similar argument, $\Gamma(A+\lambda)^{-1}u=0 $ for any $\lambda\in\C^+$.
	Using \eqref{Green} and \eqref{GreenStar}, respectively, we get that
	\beq
	\left[(A+\lambda)^{-1}-(A^*+\mu)^{-1} +(\lambda-\mu)(A^*+\mu)^{-1}(A+\lambda)^{-1}\right]u =  0
	\eeq
	and
	\beq
	\left[(A+\lambda)^{-1}-(A^*+\mu)^{-1} +(\lambda-\mu)(A+\lambda)^{-1} (A^*+\mu)^{-1}\right] u=0.
	\eeq
	Choosing $\lambda=\mub$, we get 
	\be\label{res1}
	\left[(A+\mub)^{-1}-(A^*+\mu)^{-1} +(\mub-\mu)(A^*+\mu)^{-1}(A+\mub)^{-1}\right]\vert_{H_{sa}} =  0
	\ee
	and
	\be\label{res2}
	\left[(A+\mub)^{-1}-(A^*+\mu)^{-1} +(\mub-\mu)(A+\mub)^{-1} (A^*+\mu)^{-1}\right]\vert_{H_{sa}}=0.
	\ee
	Next, let $T$ be the Cayley transform of  $A_{sa}$ at $\mu$, i.e.~$T=(A_{sa}+\mu)(A_{sa}+\mub)^{-1}= I -(\mub-\mu)(A_{sa}+\mub)^{-1}.$
	Then $(A_{sa}+\mub)^{-1} = (\mu-\mub)^{-1}(T-I)$, $(A_{sa}^*+\mu)^{-1}=(\mub-\mu)^{-1}(T^*-I)$  and \eqref{res1} is equivalent to
	\bea
	0&=& T-I +T^*-I + (T^*-I) (T-I)= T+T^*-2I + T*T -T^*-T+I= T^*T-I,
	\eea
	so $T^*T=I$. Similarly, \eqref{res2} is equivalent to $TT^*=I$, so $T$ is unitary and its inverse Cayley transform $A_{sa}$ is selfadjoint.
	\item Assume $W\subseteq H_{cns}$ is a reducing subspace such that $A\vert_{W}$ is selfadjoint. From \eqref{Green} and \eqref{GreenStar}, we get that
	\be\label{eq:g1}
	\left[(A+i)^{-1}-(A^*-i)^{-1} +2i(A^*-i)^{-1}(A+i)^{-1}\right]\vert_{W} =  -i(\Gamma(A+i)^{-1})^*(\Gamma(A+i)^{-1})\vert_{W}
	\ee
	and
	\be\label{eq:g2}
	\left[(A+i)^{-1}-(A^*-i)^{-1} +2i(A+i)^{-1} (A^*-i)^{-1}\right]\vert_{W} = -i(\Gamma_*(A^*-i)^{-1})^*(\Gamma_*(A^*-i)^{-1})\vert_{W}.
	\ee
	As we are assuming that $A\vert_{W}$ is selfadjoint, the left hand sides of \eqref{eq:g1} and \eqref{eq:g2} vanish. Due to the density of the ranges of $\Gamma$ and $\Gamma_*$, this implies that $\Gamma(A+i)^{-1}\vert_{W}=\Gamma_*(A^*-i)^{-1}\vert_{W}=0$. By the same reasoning, we get that $\Gamma(A+\lambda)^{-1}\vert_{W}=0$ for any $\lambda\in\C^+$ and $\Gamma_*(A^*+\mu)^{-1}\vert_{W}=0$ for any $\mu\in\C^-$. Thus $u\in W$ implies that
	$$u\perp\Ran\left(\Gamma(A+\lambda)^{-1}\right)^*\quad \hbox{ and }\quad u\perp \Ran\left(\Gamma_*(A^*+\mu)^{-1}\right)^*.$$ As $W$ is reducing, this also implies that for any $\lambda'\in\C^+$ 
 $$(A+\lambda')^{-1}u\perp\Ran\left(\Gamma(A+\lambda)^{-1}\right)^*\quad \hbox{ and }\quad (A+\lambda')^{-1} u\perp \Ran\left(\Gamma_*(A^*+\mu)^{-1}\right)^*,$$
 which shows that 
 $$u\perp (A^*+\overline{\lambda'})^{-1} \left(\Ran\left(\Gamma(A+\lambda)^{-1}\right)^* \bigcup  \Ran\left(\Gamma_*(A^*+\mu)^{-1}\right)^*\right).$$
 By the same reasoning for any $\mu'\in\C^-$ we have
 $$u\perp (A+\overline{\mu'})^{-1} \left(\Ran\left(\Gamma(A+\lambda)^{-1}\right)^* \bigcup  \Ran\left(\Gamma_*(A^*+\mu)^{-1}\right)^*\right),$$
 so $u\perp H_{cns}$ and $u=0$.

\item {We show that the subspace is reducing for the resolvent of $\cL$.  Let $w\in H_{sa}$. Since $w\perp H_{cns}$, we know from the Hilbert identity that for any $\lambda\in\C^-$ we have
$\Gamma(A-\lambda)^{-1}w=0=\Gamma_* (A^*-\overline{\lambda})^{-1}w$. Thus from \eqref{eq:res+} and \eqref{eq:res-}, we get 
$$(\cL-\lambda)^{-1} \left(\begin{array}{c} 0 \\ w \\ 0 \end{array}\right) = \left(\begin{array}{c} 0 \\ (A^*-{\lambda})^{-1} w \\ 0 \end{array}\right)\quad\hbox{
and }\quad
(\cL-\bar{\lambda})^{-1} \left(\begin{array}{c} 0 \\ w \\ 0 \end{array}\right) = \left(\begin{array}{c} 0 \\ (A-\overline{\lambda})^{-1} w \\ 0 \end{array}\right).$$
The claim now follows immediately from part \eqref{i}.}

 \item {Suppose 
$$ \left(\begin{array}{c} f \\ w \\ g \end{array}\right) \perp (\cL-\lambda)^{-1} \left(\begin{array}{c} L^2(\R^-,E_*) \\ 0 \\ L^2(\R^+,E) \end{array}\right) \hbox{ for all } \lambda\not\in\R.$$
Since $\cL$ is symmetric, this means that 
$$ (\cL-\lambda)^{-1}  \left(\begin{array}{c} f \\ w \\ g \end{array}\right) \perp \left(\begin{array}{c} L^2(\R^-,E_*) \\ 0 \\ L^2(\R^+,E) \end{array}\right) \hbox{ for all } \lambda\not\in\R.$$
First, let $\lambda\in\C^+$. Then by \eqref{eq:res+}, we see that 
$$ 0 = \int_x^\infty e^{i\lambda (t-x)}g(t)\ dt  = \int_0^\infty e^{i\lambda t}g(t)\ dt,$$
which implies $g\equiv0$.
Similarly, choosing   $\lambda\in\C^-$,  by \eqref{eq:res-}, we see that $f\equiv0$.}

{From the first component in \eqref{eq:res+}, we now see that $\Gamma_*(A^*-\lambda)^{-1} w=0$ and from the third component of \eqref{eq:res-}, we have $\Gamma(A-\overline{\lambda})^{-1} w=0$ for all $\lambda\in\C^+$. Thus 
$$w\perp \Ran\left(\left(\Gamma(A-\overline{\lambda})^{-1}\right)^*\right) \cup \Ran\left(\left(\Gamma_*(A^*-\lambda)^{-1}\right)^*\right).$$}

{Since $\cL$ is selfadjoint, it follows immediately from the Hilbert identity that 
$$\bigvee_{\lambda\not\in\R} (\cL-\lambda)^{-1} \left(\begin{array}{c} L^2(\R^-,E_*) \\ 0 \\ L^2(\R^+,E) \end{array}\right) $$ is a reducing subspace and therefore, for any $\mu,\lambda\not\in\R$ also
$$(\cL-\mu)^{-1}(\cL-\lambda)^{-1} \left(\begin{array}{c} 0 \\ w \\ 0 \end{array}\right)\perp  \left(\begin{array}{c} L^2(\R^-,E_*) \\ 0 \\ L^2(\R^+,E) \end{array}\right) .$$
Now, choosing $\mu,\lambda \in \C^+$, from  \eqref{eq:res-}, there exists $\tilde{g}$ such that 
$$(\cL-\mu)^{-1}(\cL-\lambda)^{-1} \left(\begin{array}{c} 0 \\ w \\ 0 \end{array}\right) = \left(\begin{array}{c} 0 \\ (A-\overline{\lambda})^{-1} w \\ \tilde{g} \end{array}\right)$$
and by repeating the arguments above, we see that $\Gamma_*(A^*-\mu)^{-1} (A-\overline{\lambda})^{-1} w=0$, giving that 
$$w\perp (A^*-\lambda)^{-1} \Ran\left(\left(\Gamma_*(A^*-\mu)^{-1}\right)^*\right).$$
Similarly, choosing $\mu$ and $\lambda$ from appropriate half-planes we see that $w\perp H_{cns}$.
Therefore, we have shown that 
\be\label{eq:Hsaperp}\left(\bigvee_{\lambda\not\in\R} (\cL-\lambda)^{-1} \left(\begin{array}{c} L^2(\R^-,E_*) \\ 0 \\ L^2(\R^+,E) \end{array}\right)\right)^\perp\subseteq \left(\begin{array}{c} 0 \\ H_{sa} \\ 0 \end{array}\right).\ee}

{On the other hand, using part \eqref{iv}, we know that 
$$ (\cL-\lambda)^{-1} \left(\begin{array}{c} 0 \\ H_{sa} \\ 0 \end{array}\right) \subseteq \left(\begin{array}{c} 0 \\ H_{sa} \\ 0 \end{array}\right). $$
Taking orthogonal complements, this gives that
$$\left(\begin{array}{c} L^2(\R^-,E_*) \\ 0 \\ L^2(\R^+,E) \end{array}\right)\subseteq \left(\begin{array}{c} L^2(\R^-,E_*) \\ H_{cns} \\ L^2(\R^+,E) \end{array}\right) \subseteq \left[(\cL-\lambda)^{-1} \left(\begin{array}{c} 0 \\ H_{sa} \\ 0 \end{array}\right)\right]^\perp$$
and so 
$$(\cL-\lambda)^{-1} \left(\begin{array}{c} L^2(\R^-,E_*) \\ 0 \\ L^2(\R^+,E) \end{array}\right) \subseteq \left[\left(\begin{array}{c} 0 \\ H_{sa} \\ 0 \end{array}\right)\right]^\perp.$$
Taking the linear span, this together with \eqref{eq:Hsaperp} gives \eqref{v}.}

\item {Since we have shown \eqref{iv} and $\cL$ is selfadjoint, it is clear that $\cL$ restricted to
$\left(\begin{array}{c} L^2(\R^-,E_*) \\ H_{cns} \\ L^2(\R^+,E) \end{array}\right)$ is a selfadjoint dilation of $A_{cns}$. Since $\cL$ is a minimal dilation of $A$ by Theorem \ref{Th:dil}, we get
\ben
\cH &=&\bigvee_{\lambda\not\in\R} (\cL-\lambda)^{-1}\left[ \left(\begin{array}{c} 0 \\ H_{cns} \\ 0 \end{array}\right)+ \left(\begin{array}{c} 0 \\ H_{sa} \\ 0 \end{array}\right)  \right]\\
&\subseteq & \bigvee_{\lambda\not\in\R} (\cL-\lambda)^{-1} \left(\begin{array}{c} 0 \\ H_{cns} \\ 0 \end{array}\right)+ \left(\begin{array}{c} 0 \\ H_{sa} \\ 0 \end{array}\right), 
\een
using part \eqref{iv}.
This proves that
$$
 \bigvee_{\lambda\not\in\R} (\cL-\lambda)^{-1} \left(\begin{array}{c} 0 \\ H_{cns} \\ 0 \end{array}\right) = \left(\begin{array}{c} L^2(\R^-,E_*) \\ H_{cns} \\ L^2(\R^+,E) \end{array}\right) 
$$
and hence the required minimality.}
\end{enumerate}
\end{proof}

The next result gives several descriptions of a core for $\cL$.

\begin{Theorem}
\begin{enumerate}
	\item The set 
$$\cC:=\left\{ \U\in D(\cL): v_+(0)\in \Ran\Gamma \right\} $$
is a core for $\cL$, i.e.~it 
is dense in $D(\cL)$ in the graph norm.
\item We have the following equivalent descriptions of $\cC$:
\ben\label{cC1}
\cC&=& \left\{ \U\in D(\cL): v_-(0)\in \Ran\Gamma_* \right\}\\
\label{cC2}
&=&  \left\{ \U\in \cH: v_\pm\in H^1, \exists h\in D(A). h-iu\in D(A^*), v_+(0)=\Gamma h, v_-(0)=\Gamma_* (h-iu) \right\}\\
\label{cC3}
&=&  \left\{ \U\in \cH: v_\pm\in H^1, \exists h_*\in D(A^*). h_*+iu\in D(A), v_-(0)=\Gamma_* h_*, v_+(0)=\Gamma (h_*+iu) \right\}
\een
{\item For $\U\in\cC$ we have 
\ben\label{cCL}
\cL\U&=& \left(\begin{array}{c} iv_-' \\ i\left[A^*(h-iu)-Ah\right] \\ iv_+' \end{array}\right),
\een
where $h\in D(A)$ is as in \eqref{cC2}.}
\end{enumerate}
\end{Theorem}

\begin{proof}
(1) Let $\perp_G$ denote orthogonality in the graph norm and $W=\W\in\cC^{\perp_G}\subseteq D(\cL)$. Choosing $U=\U\in D(\cL)$ with $v_-=0$, $u=0$ and $v_+\in H^1_0(\R^+,E)$ arbitrary, the orthogonality condition gives
$$0=\llangle g,v_+\rrangle +\llangle g',v'_+\rrangle.$$
This implies that $g'\in H^1(\R^+,E)$ and that $g-g''=0$ in the sense of distributions, so $g(x)=g(0)e^{-x}$. Similarly, choosing $v_+=0$, $u=0$ and $v_-\in H^1_0(\R^-,E_*)$ arbitrary shows that $f(x)=f(0)e^{x}$.

Next, choose $v_-=0$. Then by $(I)$ in \eqref{def:L} we have $v_+(0)=i\Gamma u$ and the orthogonality condition yields
\bea
0&=& \llangle w,u\rrangle+\llangle TW,T_*U\rrangle+\llangle g,v_+\rrangle +\llangle g',v'_+\rrangle \\
&=& \llangle w,u\rrangle+\llangle (A^*+\bar{\lambda})(w+(\Gamma(A+\lambda)^{-1})^*g(0))-\bar{\lambda}w,Au\rrangle+\llangle g(0),i\Gamma u\rrangle.
\eea
Since $u\in D(A)$ we may set $u=(A+i)^{-1}b$ for some $b\in H$. Then 
\bea
0&=&\llangle (A^*-i)^{-1}w+(A(A+i)^{-1})^*\left[(A^*+\bar{\lambda})(w+(\Gamma(A+\lambda)^{-1})^*g(0))-\bar{\lambda}w\right]-i\left(\Gamma(A+i)^{-1}\right)^*g(0),b\rrangle.
\eea
Since $b\in H$ is arbitrary and choosing $\la=i$, we get 
$$0=(A^*-i)^{-1}w+(A(A+i)^{-1})^*\left[(A^*-i)(w+(\Gamma(A+i)^{-1})^*g(0))+iw\right]-i\left(\Gamma(A+i)^{-1}\right)^*g(0).$$
Next use that 
$$(A(A+i)^{-1})^*=(I-i(A+i)^{-1})^*=I+i(A^*-i)^{-1},$$
to get 
\bea
0&=&(A^*-i)^{-1}w+(A^*-i)(w+(\Gamma(A+i)^{-1})^*g(0))+iw\\
&&+i(A^*-i)^{-1}\left[(A^*-i)(w+(\Gamma(A+i)^{-1})^*g(0))+iw\right]
-i\left(\Gamma(A+i)^{-1}\right)^*g(0)\\
&=& (A^*-i)(w+(\Gamma(A+i)^{-1})^*g(0))+2iw.
\eea
Therefore, applying $(A^*-i)^{-1}$, we get
\be\label{eq:star1}
(I+2i(A^*-i)^{-1})w=-(\Gamma(A+i)^{-1})^*g(0).
\ee

We now choose $U$ with $v_+=0$, $u=(A^*-i)^{-1}c$ and $v_-(0)=-i\Gamma_*u$. Then
\bea
0&=& \llangle w,u\rrangle+\llangle T_*W,TU\rrangle+\llangle f,v_-\rrangle +\llangle f',v'_-\rrangle \\
&=& \llangle w,u\rrangle+\llangle (A+i)(w+(\Gamma_*(A^*-i)^{-1})^*f(0))-iw,A^*u\rrangle+\llangle f(0),-i\Gamma_* u\rrangle.
\eea
A similar calculation to above shows that 
\be\label{eq:star2}
(I-2i(A+i)^{-1})w=-(\Gamma_*(A^*-i)^{-1})^*f(0).
\ee

We now go back to the case when $v_-=0$ and $u=(A+i)^{-1}b$ and write the orthogonality relation using the expression for $T_*$ rather than $T$ for $W$, i.e.~
\bea
0&=& \llangle w,u\rrangle+\llangle T_*W,T_*U\rrangle+\llangle g,v_+\rrangle +\llangle g',v'_+\rrangle \\
&=& \llangle w,u\rrangle+\llangle (A+i)(w+(\Gamma_*(A^*-i)^{-1})^*f(0))-iw,Au\rrangle+\llangle g(0),i\Gamma u\rrangle.
\eea
Using \eqref{eq:star2} and \eqref{eq:star1}, this gives
\bea 0
&=& \llangle w,u\rrangle + \llangle iw, Au\rrangle + \llangle g(0),i\Gamma u\rrangle\\ 
&=& \llangle (A^*-i)^{-1}w +i(I+i(A^*-i)^{-1})w +i(I+2i(A^*-i)^{-1})w , b\rrangle\\
&=& \llangle 2iw-2(A^*-i)^{-1}w,b\rrangle.
\eea
Thus $\left(I+i(A^*-i)^{-1}\right)w=0$, implying $w\in D(A^*)$. Applying $A^*-i$, we then find $w\in\ker(A^*)$. This implies $w\in H_{sa}$ (should we give more detail here relating it to $H_{sa}$ from previous theorem?).

On the other hand, from \eqref{eq:star1}, we now get 
$$-w=(I+2i(A^*-i)^{-1})w=-(\Gamma(A+i)^{-1})^*g(0)\in H_{cns}.$$
Hence, $w\in H_{sa}\cap H_{cns}=\{0\}$.

Equations \eqref{eq:star1} and \eqref{eq:star2} now give
\be\label{eq:star3}
(\Gamma_*(A^*-i)^{-1})^*f(0)=0=(\Gamma(A+i)^{-1})^*g(0).
\ee
Then (I) in \eqref{def:L} with $\mu=-i$ and using \eqref{eq:star3} simply becomes  $g(0)= S^*(i)f(0)$ and for any $h\in D(A)$ we have
\bea
\llangle g(0),\Gamma h\rrangle &=& \llangle f(0), S(i)\Gamma h\rrangle \ =\  \llangle f(0), \Gamma_*(A^*-i)^{-1}(A-i) h\rrangle \\
&=& \llangle (\Gamma_*(A^*-i)^{-1})^* f(0), (A-i) h\rrangle =0,
\eea
where we have again used \eqref{eq:star3}. Hence, $g(0)$ is orthogonal to $\Ran\Gamma$, which is dense in $E$, so $g(0)=0$. From $(II)$ in \eqref{def:L} we get $f(0)=0$, which completes the proof of the core property.

(2) If $v_+(0)\in\Ran\Gamma$, then from $(II)$ in \eqref{def:L}, we see that  $v_-(0)\in\Ran\Gamma_*$. The converse follows from condition $(I)$ in $D(\cL)$. Thus \eqref{cC1} holds.

To show \eqref{cC2}, first let $U=\U\in \cC$ and let 
 $w_+\in H$ be such that $v_+(0)=\Gamma (A+\la)^{-1} w_+$ for some $\la\in\C^+$.
Using \eqref{Green}, we have
\ben \label{eq:domain}
u+(\Gamma(A+\lambda)^{-1})^*v_{+}(0)&=& u+(\Gamma(A+\lambda)^{-1})^*\Gamma (A+\la)^{-1} w_+\\ \nonumber
&=& u+i\left((A+\la)^{-1}w_+-(A^*+\bar{\la})^{-1}w_++(\la-\bar{\la})(A^*+\bar{\la})^{-1}(A+\la)^{-1}w_+\right),
\een  
Thus, from the domain condition in \eqref{def:L}, we get that $u+i(A+\la)^{-1} w_+\in D(A^*)$. Set $h=(A+\la)^{-1} w_+\in D(A)$. Clearly, $h-iu\in D(A^*)$ and $v_+(0)=\Gamma h$. It remains to show that $v_-(0)=\Gamma_* (h-iu)$.
By the previous calculation, condition $(II)$ in \eqref{def:L} now becomes
\ben\label{eq:II}
v_-(0)&=& S(-\bar{\lambda})\Gamma h-i\Gamma_*\left( u+(\Gamma(A+\lambda)^{-1})^*\Gamma (A+\la)^{-1} w_+\right)\\ \nonumber
&=&S(-\bar{\lambda})\Gamma h +\Gamma_* \left[ -iu+ h-(A^*+\bar{\la})^{-1}w_++(\la-\bar{\la})(A^*+\bar{\la})^{-1}(A+\la)^{-1}w_+\right]\\ \nonumber
&=& \Gamma_* (h-iu),
\een
using the definition of $S$, \eqref{eq:S}.

On the other hand, if $U=\U$ lies in the set on the r.h.s.~of \eqref{cC2}, then clearly $v_-(0)\in \Ran\Gamma_*$ and it remains to check that $u+(\Gamma(A+\lambda)^{-1})^*v_{+}(0)\in D(A^*)$ and $(II)$ holds. Setting $v_+(0)=\Gamma h = \Gamma (A+\la)^{-1} w_+$ this follows by the same calculations as in \eqref{eq:domain} and \eqref{eq:II}. This proves \eqref{cC2}.

Finally, \eqref{cC3} follows by setting  $h_*=h-iu$.

{
(3) Let $\U\in\cC$. Then $v_+(0)=\Gamma h = \Gamma (A+\lambda)^{-1}v$ with $h$ as in \eqref{cC2} and some $v\in H$. Therefore, using \eqref{eq:T2} we have
\bea
T\U &=& \left(A^*+\mu\right)(u+(\Gamma(A+\bar{\mu})^{-1})^*v_{+}(0))-\mu u \\
&=& \left(A^*+\mu\right)(u+(\Gamma(A+\bar{\mu})^{-1})^*\Gamma (A+\lambda)^{-1}v)-\mu u.
\eea
Using \eqref{Green}, this gives
\bea
T\U &=& \left(A^*+\mu\right)(u+i ((A+\lambda)^{-1}-(A^*+\mu)^{-1} +(\lambda-\mu)(A^*+\mu)^{-1}(A+\lambda)^{-1}     )v)-\mu u\\
    &=& \left(A^*+\mu\right)(u+ih)  -iv  + i (\lambda-\mu)(A+\lambda)^{-1} v)-\mu u\\
		&=& \left(A^*+\mu\right)(u+ih)  -i(A+\lambda) h   + i (\lambda-\mu)h)-\mu u\\
		&=& i A^*(h-iu) - i Ah.
\eea
The statement now follows from the definition on $\cL$.
}
\end{proof}

\section{Discussion}

\subsection{Advantages of our Construction}
{
We compare the construction of the operators $\Gamma, \Gamma_*$ in our model to having to determine the square root of operators in other models. We consider the case when the imaginary part has finite rank: $A=\Re A +iV$ with $V$ of finite rank. Then we need to determine $\Gamma$ so that 
$$2\llangle Vu,v\rrangle=\llangle\Gamma u,\Gamma v\rrangle.$$
$V$ can be represented by a positive Hermitian matrix. Using the Cholesky decomposition, we can write $2V=\Gamma^*\Gamma$ for an upper triangular matrix $\Gamma$ with non-negative diagonal entries. Therefore, our method requires calculating the Cholesky decomposition of the matrix rather than its square root.}

\subsection{Comparison to the Kudryashov/Ryzhov model}
{Based on the work of Kudryashov, in \cite{Ryz08}, Ryzhov discusses two selfadjoint dilations (which are then shown to coincide) of a dissipative operator $A$. These are constructed using the Sz.-Nagy-Foias functional model involving square roots, as discussed at the end of Section \ref{Straus}.}

{We show that for the special choice of $\lambda=i$, our model can be recovered from the results in \cite{Ryz08}. However, the method will not reproduce our explicit formulae, as transformations that use square roots of operators are involved.}

{Let $A$ be a maximal dissipative operator, $T$ its Cayley transform \eqref{Cayley} and define $D_T$ and $D_{T^*}$ as in \eqref{DT} and \eqref{DTstar}. Set {$\widetilde{E}=\overline{\Ran D_T}\subseteq H$, $\widetilde{E}_*=\overline{\Ran D_{T^*}}\subseteq H$,} $Q=D_T/\sqrt{2}$, $Q_*=D_{T^*}/\sqrt{2}$, let $\Gammat=Q(A+i):D(A)\to H$, $\Gammat_*=Q_*(A^*-i):D(A^*)\to H$. The selfadjoint dilation $\cA$ of $A$ in \cite{Ryz08} is then given by
\begin{eqnarray}\label{def:A}
 D(\cA) & = & \left\{ \Hvec \; : \; h_0\in H,\; h_+\in H^1(\R^+,\widetilde{E}),\; h_-\in H^1(\R^-,\widetilde{E}_*), \; \right. \\
 & &\left. \;\;\; h_0+Q_*h_{-}(0)\in D(A), 
   \;\;\; h_+(0)=T^*h_-(0)+i\Gammat\left(h_0+Q_*h_{-}(0)\right) \nonumber  \right\}
\end{eqnarray}
acting in the space $L^2(\R^-,\widetilde{E}_*)\oplus H\oplus L^2(\R^+,\widetilde{E})$ with 
$$\cA\Hvec= \left(\begin{array}{c} ih'_{-} \\ -ih_0+(A+i) (h_0+Q_*h_{-}(0))\\ ih'_{+} \end{array}\right) \quad \hbox{for} \quad \Hvec\in D(\cA).$$}

{By Lemma \ref{lem:unique}, there exist unitary operators $U:\widetilde{E}\to E$ and $U_*:\widetilde{E}_*\to E_*$  such that $\Gamma=U\Gammat$ and $\Gamma_*=U_*\Gammat_*$. Here, $E,E_*,\Gamma$ and $\Gamma_*$ are as in Lemma \ref{lemma 3.1}. }

{
\begin{lem}
We have that 
$$ \threemat{U_*}{0}{0}{0}{I}{0}{0}{0}{U} \cA =\cL \threemat{U_*}{0}{0}{0}{I}{0}{0}{0}{U} .$$
\end{lem}
}
\begin{proof}{
We first show that $h_0+Q_*h_{-}(0)\in D(A)$ if and only if $u+(\Gamma_*(A^*-i)^{-1})^*v_{-}(0)\in D(A)$, where $u=h_0$ and $v_-(0)=U_*h_-(0)$. Recalling that $Q_*$ is selfadjoint, we get
\ben\label{eq:Kudryashov}
u+(\Gamma_*(A^*+\mu)^{-1})^*v_{-}(0) &=& h_0+ (U_*\Gammat_*(A^*-i)^{-1})^* U_*h_-(0) \nonumber\\
&=& h_0+ (U_*Q_*)^* U_*h_-(0)\ = \ h_0+Q_*h_{-}(0).
\een
Next we show that $h_+(0)=T^*h_-(0)+i\Gammat\left(h_0+Q_*h_{-}(0)\right) $ is equivalent to condition (I) in \eqref{def:L}.By \eqref{eq:Kudryashov}, condition (I) is equivalent to
$$Uh_+(0) = S^*(i)U_*h_-(0)+iU\Gammat(h_0+Q_*h_{-}(0)).$$
We therefore need to show that $U^*S^*(i)U_*=T^*$ or, equivalently, $S(i)=U_*TU^*$.
By \eqref{eq:S}, on $D(A)$ we have 
\bea
S(i)U\Gammat &=& S(i)\Gamma \ =\ \Gamma_*(A^*-i)^{-1}(A-i)\ =\ U_*\Gammat_*(A^*-i)^{-1}(A-i)\ =\ U_*Q_*(A-i)\\
&=& U_*Q_*(A-i)(A+i)^{-1}(A+i)\ =\ U_*Q_*T(A+i)\ =\ U_* TQ(A+i)\ =\  U_*T\Gammat,
\eea
where we have used that $Q_*T=TQ$. Thus on the dense set $\Ran\Gammat$, we have $S(i)U=U_*T$, as required.}

{
It remains to show that the action of the operators coincide. Clearly, the action on the incoming and outgoing channels coincide. That the action on the middle component coincides follows from \eqref{eq:Kudryashov} together with \eqref{eq:Tstar2}, Lemma \ref{lem:T} and Corollary \ref{cor:T}.}
\end{proof}

\subsection{{Connection of the $M$-function to characteristic function in a the case of a symmetric minimal operator} }

{In \cite{Ryz07}, Ryzhov develops a functional model for certain non-selfadjoint extensions of a symmetric operator with equal deficiency indices by using the classical boundary triple framework. We now compare the $M$-function $M(\lambda)$ arising in the boundary triple framework to our characteristic function $S(\lambda)$ in the case of an underlying symmetric operator. Related results and connections to scattering theory can be found in \cite{BMN10,BMN17}. We stress that in our construction neither symmetry of the underlying operator nor equal deficiency indices are required.
}

{Let $L$ be a symmetric operator. We construct an associated boundary triple using the von Neumann formula: $D(L^*)=D(L)\dotplus N_+\dotplus N_- $, where $N_\pm=\ker(L^*\mp i)$. Let $f=f_0+f_i+f_{-i}$ and $g=g_0+g_i+g_{-i}$ lie in $D(L^*)$ and be decomposed according to the von Neumann formula. Then 
$$\llangle L^*f,g\rrangle -\llangle f, L^* g\rrangle = 2i\left(\llangle f_i,g_i\rrangle-\llangle f_{-i},g_{-i}\rrangle\right)$$
and we get a boundary triple by choosing $\Gamma_1 f=\sqrt{2}f_i, \Gamma_0 f=\sqrt{2}f_{-i}, \Gammat_0 g=-i\sqrt{2}g_i$ and $\Gammat_1 g=-i\sqrt{2}g_{-i}$. In particular, we have
\be
\llangle L^*f,g\rrangle -\llangle f, L^* g\rrangle= i\llangle \Gamma_1 f,\Gamma_1 g\rrangle_{N_+} - i \llangle\Gamma_0 f, \Gamma_0 g\rrangle_{N_-}.
\ee
Moreover, for $B:N_-\to N_+$, let $L_B=L^*\vert_{D(L_B)}$, where $D(L_B)=\ker(\Gamma_1-B\Gamma_0)$. Then for $f,g\in D(L_B)$,
\be
\llangle L_Bf,g\rrangle -\llangle f, L_B g\rrangle= i\llangle B\Gamma_0 f,B\Gamma_0 g\rrangle_{N_+} - i \llangle\Gamma_0 f, \Gamma_0 g\rrangle_{N_-},
\ee
so $L_B$ is dissipative if and only if $B^*B\geq I_{N_-}$. For $f,g\in D(L_B^*)=\ker(\Gammat_1-B^*\Gammat_0)$, we have 
\be
\llangle L_B^*f,g\rrangle -\llangle f, L_B^* g\rrangle= i\llangle \Gammat_0 f,\Gammat_0 g\rrangle_{N_+} - i \llangle B^*\Gammat_0 f, B^*\Gammat_0 g\rrangle_{N_-},
\ee
so $L_B^*$ is anti-dissipative if and only if $BB^*\geq I_{N_+}$.
Assuming $L_B$ is maximally dissipative, we obtain the Lagrange identities \eqref{eq:Lagrange} and \eqref{eq:LagrangeStar} by choosing $\Gamma=(B^*B- I_{N_-})^{1/2}\Gamma_0$, $\Gamma_*=(BB^*- I_{N_+})^{1/2}\Gamma_1$, $E=\overline{\Ran\Gamma}$ and $E_*=\overline{\Ran\Gamma_*}$.
}

{Now, for $f=f_0+Bf_{-i}+f_{-i}\in D(L_B)$, we have
\bea
S(\lambda)(B^*B- I_{N_-})^{1/2}\Gamma_0 f  &=& S(\lambda)\Gamma f \nonumber\\
&=& \Gamma_* (L_B^*-\lambda)^{-1}(L_B-\lambda) (f_0+Bf_{-i}+f_{-i}) \nonumber\\
&=& (BB^*- I_{N_+})^{1/2}\Gamma_1 (f_0 + (L_B^*-\lambda)^{-1}((i-\lambda) Bf_{-i}- (i+\lambda)f_{-i}) \nonumber \\
&=& \frac{1}{\sqrt{2}}(BB^*- I_{N_+})^{1/2}\Gamma_1  (L_B^*-\lambda)^{-1}((i-\lambda) B\Gamma_0 f- (i+\lambda)\Gamma_0 f).
\eea
Since $\Gamma_0 u\in N_-$ is arbitrary, we get for any $f_{-i}\in N_-$ that
\be\label{eq:SRyz}
S(\lambda)(B^*B- I_{N_-})^{1/2} f_{-i} = \frac{1}{\sqrt{2}}(BB^*- I_{N_+})^{1/2}\Gamma_1 ( (L_B^*-\lambda)^{-1}((i-\lambda) Bf_{-i}- (i+\lambda)f_{-i}).
\ee
}

{
From now on, assume $f\in\ker(L^*-\lambda)$. Then 
$$(L-\lambda)f_0+(i-\lambda)f_{i}-(i+\lambda)f_{-i}=0$$
and so 
$$\Gamma_1  (L_B^*-\lambda)^{-1}((i-\lambda) f_{i}- (i+\lambda)f_{-i})=0.$$
This allows us to rewrite \eqref{eq:SRyz} as
\be\label{eq:SRyz2}
S(\lambda)(B^*B- I_{N_-})^{1/2} f_{-i} = \frac{i-\lambda}{\sqrt{2}}(BB^*- I_{N_+})^{1/2}\Gamma_1  (L_B^*-\lambda)^{-1}( Bf_{-i}- f_{i}).
\ee
Now let $\widetilde{M}(\lambda)$ be the $M$-function such that 
$$\widetilde{M}(\lambda)(\Gammat_1-B^*\Gammat_0)f=\Gammat_0 f\quad\hbox{ for }\quad f\in\ker(L^*-\lambda).$$
Then for such $f\in\ker(L^*-\lambda)$, we have
\be\label{eq:Mtilde}
\widetilde{M}(\lambda)(f_{-i}-B^*f_{i})=f_i.
\ee
Setting $w_+=Bf_{-i}- f_{i}\in N_+$ and $h=(L_B^*-\lambda)^{-1}w_+$, we get 
$$(L^*-\lambda)\left(h-\frac{w_+}{i-\lambda}\right)=0.$$
This implies 
$$\widetilde{M}(\lambda)\left(h_{-i}-B^*(h_{i}-\frac{w_+}{i-\lambda})\right)=h_i-\frac{w_+}{i-\lambda}.$$
Since $h\in D(L_B^*)$, we have $h_{-i}=B^*h_i$, and therefore we get 
$$h_i=(I+\widetilde{M}(\lambda)B^*)\frac{w_+}{i-\lambda}.$$
Inserting this in \eqref{eq:SRyz2}, we get
\be\label{eq:SRyz3}
S(\lambda)(B^*B- I_{N_-})^{1/2} f_{-i} =(BB^*- I_{N_+})^{1/2} (I+\widetilde{M}(\lambda)B^*)(Bf_{-i}- f_{i}).
\ee
From \eqref{eq:Mtilde}, we have that $(I+\widetilde{M}(\lambda)B^*)f_i=\widetilde{M}(\lambda)f_{-i}$, so
\bea
S(\lambda)(B^*B- I_{N_-})^{1/2} f_{-i} &=&(BB^*- I_{N_+})^{1/2} (B+\widetilde{M}(\lambda)B^*B-\widetilde{M}(\lambda))f_{-i}\\
&=&\left[B+(BB^*- I_{N_+})^{1/2}\widetilde{M}(\lambda)(B^*B- I_{N_-})^{1/2}\right] (B^*B- I_{N_-})^{1/2}f_{-i}.
\eea
Therefore, on $N_-$, we have
\be
S(\lambda) = B+(BB^*- I_{N_+})^{1/2}\widetilde{M}(\lambda)(B^*B- I_{N_-})^{1/2}.
\ee
\begin{remark}
\begin{enumerate}
	\item This gives a generalisation of Ryzhov's formula for the characteristic function to the case of different deficiency indices.
	\item If $L$ is not symmetric, the general formula will be much more complicated, as seen e.g.~in \eqref{SchrCharFn}.
	\item Rewriting the expressions in terms of the $M$-function $\widetilde{M}_\infty(\lambda)$ gives an easier expression. Note that $\widetilde{M}_\infty(\lambda)$ is a contraction. 
	\item Our choice of boundary triple here is not the `expected' one for an underlying symmetric operator, therefore our $M$-function is not Nevanlinna. Instead it is contractive and $L_B$ is dissipative iff $B$ is contractive.
	\item {If the deficiency indices coincide we are able to choose the $\Gamma$-operators to get Ryzhov's representation and have an $M$-function that is Nevanlinna. }
\end{enumerate}
\end{remark}}

\end{document}